\documentclass[reqno,12pt]{amsart}

\usepackage{graphicx}
\usepackage{amsmath}
\usepackage{mathabx}
\usepackage{amssymb}
\usepackage{dsfont}
\usepackage{enumitem}      

\usepackage[headings]{fullpage}

\numberwithin{equation}{section}
\def\RR{{\mathbb R}}
\def\MM{{\mathbb M}}
\def\ZZ{{\mathbb Z}}
\def\NN{{\mathbb N}}

\def\UU{{\mathbb U}}

\def\eps{\varepsilon}

\def\cupp{\mathop{\cup}}

\def\ell{l}

\def\mint{{{\bf-}\!\!\!\!\!\!\hspace{-.1em}\int}}
\def\mmint{{{\bf-}\!\!\!\!\!\hspace{-.1em}\int}}
\def\limsup{\mathop{\overline{\lim}}}
\def\liminf{\mathop{\underline{\lim}}}

\newtheorem{theorem}{Theorem}[section]
\newtheorem{lemma}[theorem]{Lemma}
\newtheorem{proposition}[theorem]{Proposition}
\newtheorem{corollary}[theorem]{Corollary}

\theoremstyle{definition}
\newtheorem{definition}[theorem]{Definition}

\theoremstyle{remark}
\newtheorem{remark}[theorem]{Remark}

\title[$\Gamma$-convergence of nonconvex integrals in Cheeger-Sobolev spaces]{\boldmath$\Gamma$\unboldmath-convergence of nonconvex integrals in Cheeger-Sobolev spaces and homogenization}

\author{\sc Omar Anza Hafsa}
\address{({\rm Omar Anza Hafsa}) UNIVERSITE DE NIMES, Laboratoire MIPA, Site des Carmes, Place Gabriel P\'eri, 30021 N\^\i mes, France and LMGC, UMR-CNRS 5508, Place Eug\`ene Bataillon, 34095 Montpellier, France.}
\email{omar.anza-hafsa@unimes.fr}

\author{Jean-Philippe Mandallena}
\address{({\rm Jean-Philippe Mandallena}) UNIVERSITE DE NIMES, Laboratoire MIPA, Site des Carmes, Place Gabriel P\'eri, 30021 N\^\i mes, France.}
\email{jean-philippe.mandallena@unimes.fr}


\begin{document}

\maketitle

\begin{abstract}
We study $\Gamma$-convergence of nonconvex variational integrals of the calculus of variations in the setting of Cheeger-Sobolev spaces. Applications to relaxation and homogenization are given.
\end{abstract}


\section{Introduction}

Let $(X,d,\mu)$ be a metric measure space, where $(X,d)$ is a length space which is complete, separable and locally compact, and $\mu$ is a positive Radon measure on $X$. Let $p>1$ be a real number and let $m\geq 1$ be an integer. Let $\Omega\subset X$ be a bounded open set  and let $\mathcal{O}(\Omega)$ be the class of open subsets of $\Omega$. In this paper we consider a family of variational integrals $E_t:W^{1,p}_\mu(\Omega;\RR^m)\times \mathcal{O}(\Omega)\to[0,\infty]$ defined by
\begin{equation}\label{Metric-Funct-1}
E_t(u,A):=\int_A L_t(x,\nabla_\mu u(x))d\mu(x),
\end{equation}
where $L_t:\Omega\times\MM\to[0,\infty]$ is a family of Borel measurable integrands  depending on a parameter $t>0$ and not necessarily convex with respect to $\xi\in\MM$, where $\MM$ denotes the space of real $m\times N$ matrices. The space $W^{1,p}_\mu(\Omega;\RR^m)$ denotes the class of $p$-Cheeger-Sobolev functions from $\Omega$ to $\RR^m$  and $\nabla_\mu u$ is the $\mu$-gradient of $u$ (see \S 3.1 for more details). 

We are concerned with the problem of computing the variational limit, in the sense of the $\Gamma$-convergence (see Definition \ref{Def-Gamma-cvg}), of the family $\{E_t\}_{t>0}$, as $t\to\infty$, to a variational integral $E_\infty:W^{1,p}_\mu(\Omega;\RR^m)\times \mathcal{O}(\Omega)\to[0,\infty]$ of the type
\begin{equation}\label{Metric-Funct-2}
E_\infty(u,A)=\int_A L_\infty(x,\nabla_\mu u(x))d\mu(x)
\end{equation}
with $L_\infty:\Omega\times\MM\to[0,\infty]$ which does not depend on the parameter $t$. When $L_\infty$ is independent of the variable $x$, the procedure of passing from \eqref{Metric-Funct-1} to \eqref{Metric-Funct-2} is referred as homogenization and was studied by many authors in the euclidean case, i.e., when the metric measure space $(X,d,\mu)$ is equal to $\RR^N$ endowed with the euclidean distance and the Lebesgue measure, see \cite{braides-defranceschi98} and the references therein. In this paper we deal with the metric measure and non-euclidean case. Such a attempt for dealing with integral representation problems of the calculus of variations in the setting of metric measure spaces was initiated in \cite{AHM15} for relaxation, see also \cite{mocanu05,HakKinPaLe14}. In fact, the interest of considering a general measure is that its support can modeled an hyperelastic structure together with its singularities like for example thin dimensions, corners, junctions, etc (for related works, see \cite{boubusep97, ansini-braides-piat99, jpm00, zhikov01,  bouchitte-fragala01, zhikov02, bouchitte-fragala02b, bouchitte-fragala02a, CZLP02, oah-jpm03, fragala03, bouchitte-fragala03, oah-jpm04, bouchitte-fragala04, jpm05, braides-piat08}). Such mechanical singular objects naturally lead to develop calculus of variations in the setting of metric measure spaces. Indeed, for example, a low multi-dimensional structures can be described by a finite number of smooth compact manifolds $S_i$ of dimension $k_i$ on which a superficial measure $\mu_i=\mathcal{H}^{k_i}|_{S_i}$ is attached. Such a situation leads to deal with the finite union of manifolds $S_i$, i.e., $X=\cup_i S_i$, together with the finite sum of measures $\mu_i$, i.e., $\mu=\sum_i\mu_i$, whose mathematical framework is that of metric measure spaces (for more examples, we refer the reader to \cite{boubusep97, zhikov02, CZLP02} and \cite[Chapter 2, \S10]{ChechkinPiatnitskiShamaev07} and the references therein).

\medskip

The plan of the paper is as follows. In the next section, we state the main results, see Theorem \ref{MainTheorem} (and Corollary \ref{Coro-MainResult}), Corollary \ref{Coro-Relax} and Theorems \ref{Coro-Homogenization} and \ref{Coro-Homogenization2}. In fact, Corollary \ref{Coro-Relax} is a relaxation result that we already proved in \cite{AHM15}. Here we obtain it by applying Theorem \ref{MainTheorem} which is a general $\Gamma$-convergence result in the $p$-growth case. Theorem \ref{Coro-Homogenization}, which is also a consequence of Theorem \ref{MainTheorem}, is a homogenization theorem of Braides-M\"uller type (see \cite{braides85,muller87}) in the setting of metric measure spaces. Note that to obtain such a metric homogenization theorem we  need to make some refinements on our general framework (see Section 2.3 and especially Definitions \ref{Def-appli-hom-1}, \ref{periodic-H-mms}, \ref{Def-appli-hom-2}, \ref{New-def-Revised-version-0}, \ref{New-def-Revised-version} and \ref{mu-PeriOdiciTY-Def}) in order to establish a subadditive theorem (see Theorem \ref{ST-MMspace}) of Ackoglu-Krengel type (see \cite{akcoglu-krengel81}). Theorem \ref{Coro-Homogenization2}, which generalizes Theorem \ref{Coro-Homogenization}, aims to deal with homogenization on low dimensional structures. In Section 3 we give the auxiliary results that we need for proving Theorem \ref{MainTheorem}. Then, Section 4 is devoted to the proof of Theorem \ref{MainTheorem}. Finally, Theorems \ref{ST-MMspace}, \ref{Coro-Homogenization} and \ref{Coro-Homogenization2} are proved in Section 5.

\subsubsection*{Notation} The open and closed balls centered at $x\in X$ with radius $\rho>0$ are denoted by:
\begin{trivlist}
\item[] $Q_\rho(x):=\Big\{y\in X:d(x,y)<\rho\Big\};$
\item[] $\overline{Q}_\rho(x):=\Big\{y\in X:d(x,y)\leq\rho\Big\}.$
\end{trivlist}
For $x\in X$ and $\rho>0$ we set 
$$
\partial Q_\rho(x):=\overline{Q}_\rho(x)\setminus Q_\rho(x)=\Big\{y\in X:d(x,y)=\rho\Big\}.
$$
For $A\subset X$, the diameter of $A$ (resp. the distance from a point $x\in X$ to the subset $A$) is defined by ${\rm diam}(A):=\sup_{x,y\in A}d(x,y)$ (resp. ${\rm dist}(x,A):=\inf_{y\in A}d(x,y)$).

The symbol $\mmint$ stands for the mean-value integral 
$$
\mint_Bf d\mu={1\over\mu(B)}\int_Bf d\mu.
$$

\section{Main results}

\subsection{The \boldmath$\Gamma$\unboldmath-convergence theorem} Here and subsequently, we assume that $\mu$ is doubling on $\Omega$, i.e., there exists a constant $C_d\geq 1$ (called doubling constant) such that 
\begin{equation}\label{D-meas}
\mu\left(Q_\rho(x)\right)\leq C_d \mu\left(Q_{\rho\over 2}(x)\right)
\end{equation}
for all $x\in \Omega$ and all $\rho>0$, and $\Omega$ supports a weak $(1,p)$-Poincar\'e inequality, i.e., there exist $C_P>0$ and $\sigma\geq 1$ such that for every $x\in \Omega$ and every $\rho>0$, 
\begin{equation}\label{1-p-PI}
\mint_{Q_\rho(x)}\left|f- \mint_{Q_\rho(x)} f d\mu\right|d\mu\leq \rho C_P\left(\mint_{Q_{\sigma\rho}(x)} g^p d\mu\right)^{1\over p}
\end{equation}
for every $f\in L^p_\mu(\Omega)$ and every $p$-weak upper gradient $g\in L^p_\mu(\Omega)$ for $f$. (For the definition of the concept of $p$-weak upper gradient, see Definition \ref{Def-p-weak-upper-gradient}.)

For each $t>0$, let $L_t:\Omega\times\MM\to[0,\infty]$ be a Borel measurable integrand. We assume that $L_t$ has $p$-growth, i.e., there exist $\alpha,\beta>0$, which do not depend on $t$, such that
\begin{equation}\label{Hyp1}
 \alpha\left|\xi\right|^p\leq L_t(x,\xi)\leq \beta\left(1+\left|\xi\right|^p\right)
\end{equation}
for all $\xi\in\MM$ and $\mu$-a.e. $x\in \Omega$.

Denote the $\Gamma$-limit inf and the $\Gamma$-limit sup of $E_t$ as $t\to\infty$ with respect to the strong convergence of $L^p_\mu(\Omega;\RR^m)$ by $\Gamma(L^p_\mu)$-$\liminf_{t\to\infty} E_t$ and $\Gamma(L^p_\mu)$-$\limsup_{t\to\infty} E_t$ which are defined by:   

\smallskip

\begin{trivlist}
\item $\displaystyle\Gamma(L^p_\mu)\hbox{-}\liminf\limits_{t\to\infty} E_t(u;A):=\inf\left\{\liminf\limits_{t\to\infty}E_t(u_t,A):u_t\stackrel{L^p_\mu}{\to} u\right\};$
\item
\item $\displaystyle\Gamma(L^p_\mu)\hbox{-}\limsup\limits_{t\to\infty} E_t(u;A):=\inf\left\{\limsup\limits_{t\to\infty}E_n(u_t,A):u_t\stackrel{L^p_\mu}{\to} u\right\}$
\end{trivlist}
for all $u\in W^{1,p}_\mu(\Omega;\RR^m)$ and all $A\in\mathcal{O}(\Omega)$.
\begin{definition}[\cite{degiorgi-franzoni75,degiorgi75}]\label{Def-Gamma-cvg}
The family $\{E_t\}_{t>0}$ of variational integrals is said to be $\Gamma(L^p_\mu)$-convergent to the variational functional $E_\infty$ as $t\to\infty$ if 
$$
\Gamma(L^p_\mu)\hbox{-}\liminf_{t\to\infty} E_t(u,A)\geq E_\infty(u,A)\geq \Gamma(L^p_\mu)\hbox{-}\limsup_{t\to\infty} E_t(u,A),
$$
for any $u\in W^{1,p}_\mu(\Omega;\RR^m)$ and any $A\in\mathcal{O}(\Omega)$, and we then write
$$
\Gamma(L^p_\mu)\hbox{-}\lim_{t\to\infty} E_t(u,A)=E_\infty(u,A).
$$
(For more details on the theory of $\Gamma$-convergence we refer to \cite{dalmaso93}.)
\end{definition}

For each $t>0$ and each $\rho>0$, let $\mathcal{H}_\mu^\rho L_t:\Omega\times \MM\to[0,\infty]$ be given by
\begin{equation}\label{t-mu-quasiconvexification}
\mathcal{H}_\mu^\rho L_t(x,\xi):=\inf\left\{\mint_{Q_\rho(x)}L_t(y,\xi+\nabla_\mu w(y))d\mu(y):w\in W^{1,p}_{\mu,0}(Q_\rho(x);\RR^m)\right\}
\end{equation}
where the space $W^{1,p}_{\mu,0}(Q_\rho(x);\RR^m)$ is the closure of 
$$
{\rm Lip}_0(Q_\rho(x);\RR^m):=\Big\{u\in{\rm Lip}(\Omega;\RR^m):u=0\hbox{ on }\Omega\setminus Q_\rho(x)\Big\} 
$$
with respect to the $W^{1,p}_\mu$-norm,  where ${\rm Lip}(\Omega;\RR^m):=[{\rm Lip}(\Omega)]^m$ with ${\rm Lip}(\Omega)$ denoting the algebra of Lipschitz functions from $\Omega$ to $\RR$. The main result of the paper is the following.
\begin{theorem}\label{MainTheorem}
If \eqref{Hyp1} holds then{\rm:}
\begin{eqnarray}
&&\Gamma(L^p_\mu)\hbox{-}\displaystyle\liminf\limits_{t\to\infty} E_t(u;A)\geq\int_A \limsup_{\rho\to 0}\liminf_{t\to\infty}\mathcal{H}^\rho_\mu L_t(x,\nabla_\mu u(x))d\mu(x);\label{MT-Eq1}\\
&&\Gamma(L^p_\mu)\hbox{-}\displaystyle\limsup\limits_{t\to\infty} E_t(u;A)=\int_A \lim_{\rho\to 0}\limsup_{t\to\infty}\mathcal{H}^\rho_\mu L_t(x,\nabla_\mu u(x))d\mu(x)\label{MT-Eq2}
\end{eqnarray}
for all $u\in W^{1,p}_\mu(\Omega;\RR^m)$ and all $A\in\mathcal{O}(\Omega)$.
\end{theorem}
As a direct consequence, we have
\begin{corollary}\label{Coro-MainResult}
If \eqref{Hyp1} holds and if 
\begin{equation}\label{Equality-Coro}
\liminf_{t\to\infty}\mathcal{H}^\rho_\mu L_t(x,\xi)=\limsup_{t\to\infty}\mathcal{H}^\rho_\mu L_t(x,\xi)
\end{equation}
for $\mu$-a.e. $x\in\Omega$, all $\rho>0$ and all $\xi\in\MM$, then
$$
\Gamma(L^p_\mu)\hbox{-}\displaystyle\lim\limits_{t\to\infty} E_t(u;A)=\int_A \lim_{\rho\to0}\lim_{t\to\infty}\mathcal{H}^\rho_\mu L_t(x,\nabla_\mu u(x))d\mu(x)
$$
for all $u\in W^{1,p}_\mu(\Omega;\RR^m)$ and all $A\in\mathcal{O}(\Omega)$. 
\end{corollary}

\subsection{Relaxation} The equality \eqref{Equality-Coro} is trivially satisfied when $L_t\equiv L$, i.e., $L_t$ does not depend on the parameter $t$. In such a case, we have 
$$
\Gamma(L^p_\mu)\hbox{-}\displaystyle\lim\limits_{t\to\infty} E_t(u;A)=\inf\left\{\liminf_{t\to\infty}\int_AL(x,\nabla_\mu u_t(x))d\mu(x):u_t\stackrel{L^p_\mu}{\to} u\right\}=:\overline{E}(u,A),
$$
i.e., the $\Gamma(L^p_\mu)$-limit  of $\{E_t\}_{t>0}$ as $t\to\infty$ is simply the $L^p_\mu$-lower semicontinuous envelope of the variational integral $\int_AL(x,\nabla_\mu u)d\mu$. Thus, the problem of computing the $\Gamma$-limit of $\{E_t\}_{t>0}$ becomes a problem of relaxation. We set 
$$
\mathcal{Q}_\mu L(x,\xi):=\lim_{\rho\to 0}\mathcal{H}^\rho_\mu L(x,\xi),
$$
where $\mathcal{H}^\rho_\mu L$ is given by \eqref{t-mu-quasiconvexification} with $L_t$ replaced by $L$, and we naturally call $\mathcal{Q}_\mu L$ the $\mu$-quasiconvexification of $L$. Then, Corollary \ref{Coro-MainResult} implies the following result. 
\begin{corollary}\label{Coro-Relax}
If \eqref{Hyp1} holds then
$$
\overline{E}(u,A)=\int_A\mathcal{Q}_\mu L(x,\nabla_\mu u(x))d\mu(x)
$$
for all $u\in W^{1,p}_\mu(\Omega;\RR^m)$ and all $A\in\mathcal{O}(\Omega)$. 
\end{corollary}
We thus retrieve \cite[Corollary 2.29]{AHM15}.

\subsection{Homogenization}

In order to apply Theorem \ref{MainTheorem} (and Corollary \ref{Coro-MainResult}) to homogenization, it is necessary to make some refinements on our general setting. These refinements are a first attempt to develop a framework for dealing with homogenization of variational integrals of the calculus of variations in metric measure spaces. 

\medskip

We begin with the following five definitions (see Definition \ref{Def-appli-hom-1} together with Definitions \ref{periodic-H-mms}-\ref{Def-appli-hom-2} and Definitions \ref{New-def-Revised-version-0}-\ref{New-def-Revised-version}) which set a framework to deal with homogenization of variational integrals in Cheeger-Sobolev spaces. Let  ${\rm Homeo}(X)$ be the group of homeomorphisms on $X$ and let $\mathcal{B}(X)$ be the class of Borel subsets of $X$. 

\begin{definition}\label{Def-appli-hom-1}
The metric measure space $(X,d,\mu)$ is called a $(G,\{h_t\}_{t>0})$-metric measure space if it is endowed with a pair $(G,\{h_t\}_{t>0})$, where $G$ and $\{h_t\}_{t>0}$ are subgroups of ${\rm Homeo}(X)$, such that: 
\begin{enumerate}[leftmargin=*]
\item[(a)] the measure $\mu$ is {\em $G$-invariant}, i.e., $g^\sharp\mu=\mu$ for all $g\in G$;
\item[(b)] there exists $\UU\in\mathcal{B}(X)$, which is called the unit cell, such that $\mu\big(\mathring{\UU}\big)\in]0,\infty[$ and $\mu(\partial \UU)=0$ with $\partial \UU=\overline{\UU}\setminus\mathring{\UU}$;
\item[(c)] the family $\{h_t\}_{t>0}$ of homeomorphisms on $X$ is such that: 
\begin{eqnarray}
&&\hbox{$h_1={\rm id}_X$;}\label{Homeo-1}\\
&&\hbox{$h_{st}=h_s{\hskip0.3mm\rm o\hskip0.3mm}h_t\hbox{ for all }s,t>0;$}\label{Homeo-2}\\
&&\hbox{$h_t^\sharp\mu=\mu(h_t(\UU))\mu$ for all $t>0$}.\label{Eq-sub-im}
\end{eqnarray}
\end{enumerate}
\end{definition}

\begin{remark}
Assuming that $(X, d,\mu)$ is a $(G,\{h_t\}_{t>0})$-metric measure space, it is easy to see that
\begin{equation}\label{Eq-sub-1-bis}
\mu(h_{st}(\UU))=\mu(h_s(\UU))\mu(h_t(\UU))
\end{equation}
for all $s,t>0$. In particular, as $\mu(\UU)\not=0$ we have $\mu(h_t(\UU))\not=0$ for all $t>0$, and so we see that $\mu(\UU)=1$ by using \eqref{Eq-sub-im}.
\end{remark}

\begin{definition}\label{periodic-H-mms}
When  $(X, d,\mu)$ is a $(G,\{h_t\}_{t>0})$-metric measure space, we say that  $(X, d,\mu)$ is {\em meshable} if for each $i\in\NN^*$ and each $k\in\NN^*$ there exists a finite subset $G^k_i$ of $G$ such that $(g{\hskip0.3mm\rm o\hskip0.3mm}h_k(\UU))_{g\in G^k_i}$ is a disjointed finite family 
and
\begin{equation}\label{Eq-sub-6}
h_{ik}(\UU)=\cupp_{g\in G^k_i}g{\hskip0.3mm\rm o\hskip0.3mm}h_k(\UU).
\end{equation}
\end{definition}

\begin{remark}\label{Remark-Meshable}
It is easily seen that a $(G,\{h_t\}_{t>0})$-metric measure space $(X,d,\mu)$ is meshable if and only if for each $i\in\NN^*$ and each $k\in\NN^*$ there exists a finite subset $G^k_i$ of $G$ such that $(g{\hskip0.3mm\rm o\hskip0.3mm}h_k(\UU))_{g\in G^k_i}$ is a disjointed finite family of subsets of $h_{ik}(\UU)$ and
\begin{equation}\label{Equa-Sub-1-Bis-Bis}
{\rm card}(G^k_i)=\mu(h_i(\UU)).
\end{equation}
In particular, the cardinal of $G^k_i$ does not depend on $k$. (Here and in what follows, $\NN^*$ denotes the set of integers greater than $1$.)
\end{remark}

\begin{remark}\label{Remark-RN-Asymptotically-periodic}
When $X=\RR^N$ is endowed with the euclidean distance $d_2$ and the Lebesgue measure $\mathcal{L}_N$, we consider $G\equiv\ZZ^N$, $\UU=[0,1[^N=:Y$  and $\{h_t\}_{t>0}$ given by $h_t:\RR^N\to\RR^N$ defined by $h_t(x)=tx$. In this case, for each $i\in\NN^*$ and each $k\in\NN^*$, we have 
$$
G^k_i=\Big\{(kn_1,kn_2,\cdots,kn_N):n_j\in\{0,\cdots,i-1\}\hbox{ with }j\in\{1,\cdots,N\}\Big\}.
$$ 
Note that $G^k_i=kG^1_i$ and so ${\rm card}(G^k_i)$ does not depend on $k$. More precisely, we have ${\rm card}(G^k_i)=i^N=\mathcal{L}_N(h_i(Y))$. In addition, $(\RR^N,d_2,\mathcal{L}_N)$ is meshable.
 \end{remark}

In what follows, $\mathcal{F}(X)$ denotes an arbitrary subclass of $\mathcal{B}(X)$.

\begin{definition}\label{Def-appli-hom-2}
When $(X, d,\mu)$ is a meshable $(G,\{h_t\}_{t>0})$-metric measure space, we say that $(X, d,\mu)$ is {\em asymptotically periodic with respect to} $\mathcal{F}(X)$ if for each  $A\in \mathcal{F}(X)$ and for each $k\in\NN^*$ there exists $t_{A,k}>0$ such that for each $t\geq t_{A,k}$, there exist $k^{-}_{t},k^{+}_{t}\in\NN^*$ and $g^{-}_{t},g^{+}_{t}\in G$ such that: 
\begin{eqnarray} 
&& g^-_{t}{\hskip0.3mm\rm o\hskip0.3mm}h_{kk^-_{t}}(\UU)\subset h_t(A)\subset g^+_{t}{\hskip0.3mm\rm o\hskip0.3mm}h_{kk^+_{t}}(\UU);\label{Hyp-Sub-2}\\
&& \lim_{t\to\infty}{\mu\big(h_{k^+_{t}}(\UU)\big)\over\mu\big(h_{k^-_{t}}(\UU)\big)}=1.\label{Hyp-Sub-1}
\end{eqnarray}
\end{definition}

\begin{remark}\label{Rajout-Revision-Remark}
For $(X,d,\mu)\equiv(\RR^N,d_2,\mathcal{L}_N)$ we consider $G\equiv \ZZ^N$, $\UU=Y$  and $\{h_t\}_{t>0}$ given by $h_t:\RR^N\to\RR^N$ defined by $h_t(x)=tx$ (see Remark \ref{Remark-RN-Asymptotically-periodic}). In particular, we have $g{\hskip0.3mm\rm o\hskip0.3mm}h_k(Y)=kY+g$ for all $k\in\NN^*$ and all $g\in G$. Then $(\RR^N,d_2,\mu)$ is asymptotically periodic with respect to ${\rm Cub}(\RR^N)$, where ${\rm Cub}(\RR^N)$ is the class of open cubes $C$ of $\RR^N$. 

Indeed, if $C=\prod_{i=1}^N]a_i,b_i[$ with $c=b_1-a_1=\cdots=b_N-a_N>0$ and if $k\in\NN^*$, then for every $t\geq {2k\over c}$, \eqref{Hyp-Sub-2} is satisfied  with: 
\begin{trivlist}
\item[] $k^-_t=\big[{t c\over k}\big]-1$ and $k_t^+=\big[{t c\over k}\big]+1$;
\item[] $g_t^-=k(z_t+\hat e)$ and $g^+_t=kz_t$ where $\hat e=(1,\cdots,1)$ and $z=(z_{t}^1,\cdots,z_{t}^N)$ with $z_{t}^i=\big[{ta_i\over k}\big]$  for all $i\in\{1,\cdots, N\}$,
\end{trivlist}
where $\big[x\big]$ denotes the integer part of the real number $x$. Moreover, for such $k^-_t$ and $k^+_t$, it is easily seen that \eqref{Hyp-Sub-1} is verified. 

Nevertheless, $(\RR^N,d_2,\mathcal{L}_N)$ is not asymptotically periodic with respect to ${\rm Ba}(\RR^N)$, where ${\rm Ba}(\RR^N)$ is the class of open balls (with respect to $d_2$) of $\RR^N$.
\end{remark}

In light of Remark \ref{Rajout-Revision-Remark} we introduce another ``weak" notion of ``asymptotic periodicity" together with another ``strong" notion of ``meshability", see Definitions \ref{New-def-Revised-version} and \ref{New-def-Revised-version-0} below which plays the role of Definitions \ref{periodic-H-mms} and \ref{Def-appli-hom-2} (see also Remark \ref{Rajout-Revision-Remark-Bis}). 

\begin{definition}\label{New-def-Revised-version-0}
When $(X, d,\mu)$ is a $(G,\{h_t\}_{t>0})$-metric measure space, we say that $(X, d,\mu)$ is {\em strongly meshable} if  the following four assertions are satisfied:
\begin{enumerate}[leftmargin=*]
\item[(a)] for each finite subset $H$ of $G$, the family $(g(\UU))_{g\in H}$ is finite and disjointed;
\item[(b)] if $H_1$ and $H_2$ are two finite subsets of $G$ such that 
$
\cupp_{g\in H_1}g(\UU)\subset \cupp_{g\in H_2}g(\UU),
$
then $H_1\subset H_2$ and 
$$
\cupp\limits_{g\in H_2}g(\UU)=\bigg(\cupp\limits_{g\in H_1}g(\UU)\bigg)\cup\bigg(\cupp\limits_{g\in H_2\setminus H_1}g(\UU)\bigg);
$$
\item[(c)] for each $i\in\NN^*$ and each $f\in G$ there exists a finite subset $G_i(f)$ of $G$ such that 
$$
f{\hskip0.3mm\rm o\hskip0.3mm}h_{i}(\UU)=\cupp_{g\in G_i(f)}g(\UU);
$$
\item[(d)] for each finite subset $H$ of $G$, there exist $i_H\in\NN^*$ and $f_H\in G$ such that 
$$
\cupp_{g\in H}g(\UU)\subset f_H{\hskip0.3mm\rm o\hskip0.3mm}h_{i_H}(\UU).
$$
\end{enumerate}
\end{definition}

\begin{remark}
The metric measure space $(\RR^N,d_2,\mathcal{L}_N)$ where $G\equiv \ZZ^N$, $\UU=Y$  and $\{h_t\}_{t>0}\equiv\{tx\}_{t>0}$ is strongly meshable with
$$
G_i(z)=\Big\{z+(n_1,n_2,\cdots,n_N):n_j\in\{0,\cdots,i-1\}\hbox{ with }j\in\{1,\cdots,N\}\Big\}
$$
for all $i\in\NN^*$ and all $z\in \ZZ^N$. 
\end{remark}

\begin{definition}\label{New-def-Revised-version}
When $(X, d,\mu)$ is a strongly meshable $(G,\{h_t\}_{t>0})$-metric measure space, we say that $(X, d,\mu)$ is {\em weakly asymptotically periodic with respect to} $\mathcal{F}(X)$ if for each $A\in \mathcal{F}(X)$, each $k\in\NN^*$ and each $t>0$, there exist finite subsets $G^-_{t,k}$ and $G^+_{t,k}$ of $G$ such that the families $(g{\hskip0.3mm\rm o\hskip0.3mm}h_k(\UU))_{g\in G^-_{t,k}}$ and $(g{\hskip0.3mm\rm o\hskip0.3mm}h_k(\UU))_{g\in G^+_{t,k}}$ are disjointed and satisfy the following two properties: 
\begin{eqnarray}
&&\cupp\limits_{g\in G^-_{t,k}}g{\hskip0.3mm\rm o\hskip0.3mm}h_k(\UU)\subset h_t(A)\subset\cupp\limits_{g\in G^+_{t,k}}g{\hskip0.3mm\rm o\hskip0.3mm}h_k(\UU);\label{Hyp-Sub-1-revised}\\
&&\lim_{t\to\infty}{\mu\bigg(\cupp\limits_{g\in G^+_{t,k}}g{\hskip0.3mm\rm o\hskip0.3mm}h_k(\UU)\setminus \cupp\limits_{g\in G^-_{t,k}}g{\hskip0.3mm\rm o\hskip0.3mm}h_k(\UU)\bigg)\over\mu(h_t(A))}=0.\label{Hyp-Sub-3-revised}
\end{eqnarray}
\end{definition}

\begin{remark}\label{Rajout-Revision-Remark-Bis}
From Nguyen and Zessin \cite[Lemma 3.1]{nguyen-zessin79} (see also \cite[Lemma 2.2]{licht-michaille02}) we see that for $(X,d,\mu)\equiv(\RR^N,d_2,\mathcal{L}_N)$ with $G\equiv \ZZ^N$, $\UU=Y$  and $\{h_t\}_{t>0}\equiv\{tx\}_{t>0}$, Definition \ref{New-def-Revised-version} is satisfied with  $\mathcal{F}(X)\equiv {\rm Conv}_{\rm b}(\RR^N)$, where ${\rm Conv}_{\rm b}(\RR^N)$ denotes the class of bounded Borel convex subsets of $\RR^N$. In this case, for each $A\in {\rm Conv}_{\rm b}(\RR^N)$, each $k\in\NN^*$ and each $t>0$, we have:
\begin{trivlist}
\item $G^-_{t,k}=\Big\{z\in k\ZZ^N:z+kY\subset tA\Big\}$;
\item $G^+_{t,k}=\Big\{z\in k\ZZ^N: (z+kY)\cap tA\not=\emptyset\Big\}$.
\end{trivlist}
Thus, $(\RR^N,d_2,\mathcal{L}_N)$ is weakly asymptotically periodic with respect to ${\rm Ba}(\RR^N)$ and ${\rm Cub}(\RR^N)$.
\end{remark}

In the framework of a $(G,\{h_t\}_{t>0})$-metric measure space (see Definition \ref{Def-appli-hom-1}) which is either meshable and asymptotically periodic (Definitions \ref{periodic-H-mms} and \ref{Def-appli-hom-2}) or strongly meshable and weakly asymptotically periodic (see Definitions \ref{New-def-Revised-version-0} and \ref{New-def-Revised-version}), we can establish a subadditive theorem, see Theorem \ref{ST-MMspace}, of Ackoglu-Krengel type (see \cite{akcoglu-krengel81}). Let $\mathcal{B}_0(X)$ denote  the class of Borel subsets $A$ of $X$ such that $\mu(A)<\infty$ and $\mu(\partial A)=0$ with $\partial A=\overline{A}\setminus\mathring{A}$. We first recall the definition of a subadditive (with respect to the disjointed union) and $G$-invariant set function.
\begin{definition}
Let $\mathcal{S}:\mathcal{B}_0(X)\to[0,\infty]$ be a set function.
\begin{enumerate}[leftmargin=*]
\item[(a)] The set function $\mathcal{S}$ is said to be subadditive (with respect to the disjointed union) if 
$$
\mathcal{S}(A\cup B)\leq\mathcal{S}(A)+\mathcal{S}(B)
$$ 
for all $A,B\in\mathcal{B}_0(X)$ such that $A\cap B=\emptyset$. 
\item[(b)] Given a subgroup $G$ of ${\rm Homeo}(X)$, the set function $\mathcal{S}$ is said to be $G$-invariant if
$$
\mathcal{S}\big(g(A)\big)=\mathcal{S}(A)
$$
for all $A\in \mathcal{B}_0(X)$ and all $g\in G$.
\end{enumerate}
\end{definition}
The following result, which is proved in Section 5, will be used in the proof of Theorems \ref{Coro-Homogenization} and \ref{Coro-Homogenization2} below. In what follows $\mathfrak{S}(X)$ denotes a subclass of $\mathcal{B}_0(X)$.
\begin{theorem}\label{ST-MMspace}
Assume that $(X, d,\mu)$ is a $(G,\{h_t\}_{t>0})$-metric measure space which is either meshable and asymptotically periodic or strongly meshable and weakly asymptotically periodic with respect to $\mathfrak{S}(X)$ and $\mathcal{S}:\mathcal{B}_0(X)\to[0,\infty]$ is a subadditive and $G$-invariant set function with the following property{\rm:}
\begin{equation}\label{Sub4}
\mathcal{S}(A)\leq c\mu(A)
\end{equation}
for all $A\in\mathcal{B}_0(X)$ and some $c>0$. Then
$$
\lim_{t\to\infty}{\mathcal{S}\big(h_t(Q)\big)\over\mu\big(h_t(Q)\big)}=\inf_{k\in\NN^*}{\mathcal{S}\left(h_k(\UU)\right)\over\mu(h_k(\UU))}
$$
for all $Q\in \mathfrak{S}(X)$.
\end{theorem}

Let $L:X\times\MM\to[0,\infty]$ be a Borel measurable integrand assumed to be $G$-invariant, i.e., for $\mu$-a.e. $x\in X$ and every $\xi\in \MM$,
$
L(g(x),\xi)=L(x,\xi)\hbox{ for all }g\in G. 
$
For each $t>0$, Let $L_t:X\times\MM\to[0,\infty]$ be given by 
\begin{equation}\label{mms-periodicity}
L_t(x,\xi)=L(h_t(x),\xi).
\end{equation}
(Note that $\{L_t\}_{t>0}$ is then $(G,\{h_t\}_{t>0})$-periodic, i.e., $L_t((h^{-1}_t{\rm o}g{\rm o}h_t)(x),\xi)=L_t(x,\xi)$ for all $x\in X$, all $\xi\in\MM$, all $t>0$ and all $g\in G$.) 
 
 For convenience, we introduce the following definition.
 \begin{definition}\label{mu-PeriOdiciTY-Def}
 Such a $\{L_t\}_{t>0}$, defined by \eqref{mms-periodicity}, is called  a family of {\em $(G,\{h_t\}_{t>0})$-periodic integrands modelled on $L$}.
 \end{definition}
 \begin{remark}
 If $(X,d,\mu)\equiv(\RR^N,d_2,\mathcal{L}_N)$ with $G\equiv \ZZ^N$, $\UU=Y$ and $\{h_t\}_{t>0}\equiv\{tx\}_{t>0}$, then $G$-periodicity is $Y$-periodicity and  $(G,\{h_t\}_{t>0})$-periodicity corresponds to ${1\over t}Y$-periodicity.
 \end{remark}
 
 Let ${\rm Ba}(X)$ be the class of open balls $Q$ of $X$ such that $\mu(\partial Q)=0$, where  $\partial Q:=\overline{Q}\setminus Q$. (Then ${\rm Ba}(X)\subset\mathcal{B}_0(X)$.) Applying Corollary \ref{Coro-MainResult} we then have
 \begin{theorem}\label{Coro-Homogenization}
Assume that $(X, d,\mu)$ is a $(G,\{h_t\}_{t>0})$-metric measure space which is either meshable and asymptotically periodic  or strongly meshable and  weakly asymptotically periodic with respect to ${\rm Ba}(X)$. If \eqref{Hyp1} holds and if $\{L_t\}_{t>0}$ is a family of $(G,\{h_t\}_{t>0})$-periodic integrands modelled on $L$ then
$$
\Gamma(L^p_\mu)\hbox{-}\displaystyle\lim\limits_{t\to\infty} E_t(u;A)=\int_{A}L_{\rm hom}(\nabla_\mu u(x))d\mu(x)
$$
for all $u\in W^{1,p}_\mu(\Omega;\RR^m)$ and all $A\in\mathcal{O}(\Omega)$ with $L^{\rm hom}:\MM\to[0,\infty]$ given by
$$
 L_{\rm hom}(\xi):=\inf_{k\in\NN^*}\inf\left\{\mint_{h_k\left(\mathring{\UU}\right)}L(y,\xi+\nabla_\mu w(y))d\mu(y):w\in W^{1,p}_{\mu,0}\left(h_k\big(\mathring{\UU}\big);\RR^m\right)\right\}.
$$
\end{theorem}

Theorem \ref{Coro-Homogenization} can be applied when $X$ is a $N$-dimensional manifold diffeomorphic to $\RR^N$.  In such a case, we have $d(\cdot,\cdot)=d_2(\Psi^{-1}(\cdot),\Psi^{-1}(\cdot))$, $\mu=({\Psi^{-1}})^\sharp\mathcal{L}_N$,
$\UU=\Psi(Y)$, $G\equiv\Psi(\ZZ^N)$  and $\{h_t\}_{t>0}\subset{\rm Homeo}(X)$ is given by $h_t(x)=\Psi(t\Psi^{-1}(x))$, where $\Psi$ is the corresponding diffeomorphism from $\RR^N$ to $X$. Moreover, Theorem \ref{Coro-Homogenization} can be generalized as follows.

\begin{theorem}\label{Coro-Homogenization2}
Assume that there exists a finite family $\{X_i\}_{i\in I}$ of subsets of $X$ such that
$
X=\cupp_{i\in I} X_i
$
and $\mu(X_i\cap X_j)=0$ for all $i\not=j$ and for which every $(X_i, d_{|X_i})$ is  a complete, separable and locally compact length space and every $(X_i, d_{|X_i},\mu_{|X_i})$ is a  $(G_i,\{h_t^i\}_{t>0})$-metric measure space which is either meshable and asymptotically periodic  or strongly meshable and weakly asymptotically periodic with respect to ${\rm Ba}(X_i)$, where $G_i$ and $\{h_t^i\}_{t>0}$ are subgroups of ${\rm Homeo}(X_i)$. Let $\{L_t\}_{t>0}$  be given by 
$$
L_t(x,\cdot):=L_{t}^i(x,\cdot)\hbox{ if } x\in X_i,
$$
where every $\{L_t^i\}_{t>0}$ is a family of $(G_i,\{h_t^i\}_{t>0})$-periodic integrands modelled on $L^i$.  If 
$
\Omega=\cup_{i\in I}\Omega_i
$
with every $\Omega_i\subset X_i$ being an open set and if \eqref{Hyp1} holds then
$$
\Gamma(L^p_\mu)\hbox{-}\displaystyle\lim\limits_{t\to\infty} E_t(u;A)=\sum_{i\in I}\int_{\Omega_i\cap A}L^i_{\rm hom}(\nabla_\mu u(x))d\mu(x)
$$
for all $u\in W^{1,p}_\mu(\Omega;\RR^m)$ and all $A\in\mathcal{O}(\Omega)$, where every $L_{\rm hom}^i:\MM\to[0,\infty]$ is given by 
\begin{equation}\label{HomoGenizaTion-Formula-LDS}
L_{\rm hom}^i(\xi):=\inf_{k\in\NN^*}\inf\left\{\mint_{h_k^i\left(\mathring{\UU}_i\right)}L^{i}(y,\xi+\nabla_\mu w)d\mu:w\in W^{1,p}_{\mu,0}(h_k^i\big(\mathring{\UU}_i\big);\RR^m)\right\}
\end{equation}
with $\UU_i$ denoting the unit cell in $X_i$.
\end{theorem}


\section{Auxiliary results}

\subsection{The \boldmath$p$\unboldmath-Cheeger-Sobolev spaces} Let $p>1$ be a real number, let $(X,d,\mu)$ be a metric measure space, where $(X,d)$ is a length space which is complete, separable and locally compact, and $\mu$ is a positive Radon measure on $X$, and let $\Omega\subset X$ be a bounded open set. We begin with the concept of upper gradient introduced by Heinonen and Koskela (see \cite{heinonen-koskela98}). 
\begin{definition}
A Borel function $g:\Omega\to[0,\infty]$ is said to be an upper gradient for $f:\Omega\to\RR$  if 
$
|f(c(1))-f(c(0))|\leq\int_0^1 g(c(s))ds
$
for all continuous rectifiable curves $c:[0,1]\to \Omega$. 
\end{definition}
The concept of upper gradient has been generalized by Cheeger as follows (see \cite[Definition 2.8]{cheeger99}). 
\begin{definition}\label{Def-p-weak-upper-gradient}
A function $g\in L^p_\mu(\Omega)$ is said to be a $p$-weak upper gradient for $f\in L^p_\mu(\Omega)$ if there exist $\{f_n\}_n\subset L^p_\mu(\Omega)$ and $\{g_n\}_n\subset L^p_\mu(\Omega)$ such that for each $n\geq 1$, $g_n$ is an upper gradient for $f_n$, $f_n\to f$ in $L^p_\mu(\Omega)$ and $g_n\to g$ in $L^p_\mu(\Omega)$. 
\end{definition}
Denote the algebra of Lipschitz functions from $\Omega$ to $\RR$ by ${\rm Lip}(\Omega)$. (Note that, by Hopf-Rinow's theorem (see \cite[Proposition 3.7, p. 35]{Bridson-Haefliger99}), the closure of $\Omega$ is compact, and so every Lipschitz function from $\Omega$ to $\RR$ is bounded.) From Cheeger and Keith (see \cite[Theorem 4.38]{cheeger99} and \cite[Definition 2.1.1 and Theorem 2.3.1]{keith1-04}) we have

\begin{theorem}\label{cheeger-theorem}
If $\mu$ is doubling on $\Omega$, i.e., \eqref{D-meas} holds, and $\Omega$ supports a weak $(1,p)$-Poincar\'e inequality, i.e., $\eqref{1-p-PI}$ holds, then there exists a countable family $\{(\Omega_\alpha,\xi^\alpha)\}_\alpha$ of $\mu$-measurable disjoint subsets $\Omega_\alpha$ of $\Omega$ with $\mu(\Omega\setminus\cup_\alpha \Omega_\alpha)=0$ and of functions $\xi^\alpha=(\xi^\alpha_1,\cdots,\xi^\alpha_{N(\alpha)}):\Omega\to\RR^{N(\alpha)}$ with $\xi^\alpha_i\in{\rm Lip}(\Omega)$ satisfying the following properties{\rm:}
\begin{enumerate}[leftmargin=*]
\item[\rm(a)]  there exists an integer $N\geq 1$ such that $N(\alpha)\in\{1,\cdots, N\}$ for all $\alpha;$ 
\item[\rm(b)] for every $\alpha$ and every $f\in{\rm Lip}(\Omega)$ there is a unique $D_\mu^\alpha f\in L^\infty_\mu(\Omega_\alpha;\RR^{N(\alpha)})$ such that for $\mu$-a.e. $x\in \Omega_\alpha$,
$$
\lim_{\rho\to 0}{1\over\rho}\|f-f_x\|_{L^\infty_{\mu}(Q_\rho(x))}=0,
$$
where $f_x\in {\rm Lip}(\Omega)$ is given by $f_x(y):=f(x)+D_\mu^\alpha f(x)\cdot(\xi^\alpha(y)-\xi^\alpha(x));$ in particular 
$$
D_\mu^\alpha f_x(y)=D_\mu^\alpha f(x)\hbox{ for $\mu$-a.e. $y\in \Omega_\alpha$};
$$
\item[\rm(c)] the operator $D_\mu:{\rm Lip}(\Omega)\to L^\infty_\mu(\Omega;\RR^N)$ given by
$$
D_\mu f:=\sum_\alpha \mathds{1}_{X_\alpha}D_\mu^\alpha f,
$$
where $\mathds{1}_{\Omega_\alpha}$ denotes the characteristic function of $\Omega_\alpha$, is linear and, for each $f,g\in{\rm Lip}(\Omega)$, one has 
$$
D_\mu(fg)=fD_\mu g+gD_\mu f;
$$
\item[\rm(d)] for every $f\in{\rm Lip}(\Omega)$, $D_\mu f=0$ $\mu$-a.e. on every $\mu$-measurable set where $f$ is constant.
\end{enumerate}
\end{theorem}

\begin{remark}
Theorem \ref{cheeger-theorem} is true without the assumption that $(X,d)$ is a length space.
\end{remark}

Let ${\rm Lip}(\Omega;\RR^m):=[{\rm Lip(\Omega)}]^m$ and let $\nabla_\mu:{\rm Lip}(\Omega;\RR^m)\to L^\infty_\mu(\Omega;\MM)$ given by 
$$
\nabla_\mu u:=\left(
\begin{array}{c}
D_\mu u_1\\
\vdots\\ 
D_\mu u_m
\end{array}
\right)
\hbox{ with }u=(u_1,\cdots,u_m).
$$

From Theorem \ref{cheeger-theorem}(c) we see that for every $u\in {\rm Lip}(\Omega;\RR^m)$ and every $f\in {\rm Lip}(\Omega)$, one has
\begin{equation}\label{Mu-der-Prod}
\nabla_\mu (fu)=f\nabla_\mu u+D_\mu f\otimes u.
\end{equation}
\begin{definition} 
The $p$-Cheeger-Sobolev space $W^{1,p}_\mu(\Omega;\RR^m)$  is defined as the completion of ${\rm Lip}(\Omega;\RR^m)$ with respect to the norm
\begin{equation}\label{W1pmu-norm}
\|u\|_{W^{1,p}_\mu(\Omega;\RR^m)}:=\|u\|_{L^p_\mu(\Omega;\RR^m)}+\|\nabla_\mu u\|_{L_\mu^p(\Omega;\MM)}.
\end{equation}
\end{definition}
Taking Proposition \ref{Fundamental-Proposition-for-CalcVar-in-MMS}(a) below into account, since $\|\nabla_\mu u\|_{L_\mu^p(\Omega;\MM)}\leq \|u\|_{W^{1,p}_\mu(\Omega;\RR^m)}$ for all $u\in{\rm Lip}(\Omega;\RR^m)$ the linear map $\nabla_\mu$ from ${\rm Lip}(\Omega;\RR^m)$ to $L_\mu^p(\Omega;\MM)$  has a unique extension to $W^{1,p}_\mu(\Omega;\RR^m)$ which will still be denoted by $\nabla_\mu$ and will be called the $\mu$-gradient.

\begin{remark}
When $\Omega$ is a bounded open subset of $X=\RR^N$ and $\mu$ is the Lebesgue measure on $\RR^N$, we retrieve the (classical) Sobolev spaces $W^{1,p}(\Omega;\RR^m)$.  For more details on the various possible extensions of the classical theory of the Sobolev spaces to the setting of metric measure spaces, we refer to \cite[\S 10-14]{heinonen07} (see also \cite{cheeger99, shanmugalingam00, gol-tro01,hajlasz02}).
\end{remark}

The following proposition (whose proof is given below, see also \cite[Proposition 2.28]{AHM15}) provides useful properties for dealing with calculus of variations in the metric measure setting.

\begin{proposition}\label{Fundamental-Proposition-for-CalcVar-in-MMS}
Under the hypotheses of Theorem {\rm\ref{cheeger-theorem}}, we have{\rm:}
\begin{enumerate}[leftmargin=*]
\item[\rm(a)] the $\mu$-gradient is closable in $W^{1,p}_\mu(\Omega;\RR^m)$, i.e., for every $u\in W^{1,p}_\mu(\Omega;\RR^m)$ and every $A\in\mathcal{O}(\Omega)$, if $u(x)=0$ for $\mu$-a.e. $x\in A$ then $\nabla_\mu u(x)=0$ for $\mu$-a.e. $x\in A;$
\item[\rm(b)] $\Omega$ supports a $p$-Sobolev inequality, i.e., there exist $C_S>0$ and $\chi\geq 1$ such that
\begin{equation}\label{Poincare-Inequality}
\left(\int_{Q_\rho(x)}|v|^{\chi p}d\mu\right)^{1\over\chi p}\leq \rho C_S\left(\int_{Q_\rho(x)}|\nabla_\mu v|^pd\mu\right)^{1\over p}
\end{equation}
for all $0<\rho\leq \rho_0$, with $\rho_0>0$, and all $v\in W^{1,p}_{\mu,0}(Q_\rho(x);\RR^m)$, where, for each $A\in\mathcal{O}(\Omega)$, $W^{1,p}_{\mu,0}(A;\RR^m)$ is the closure of ${\rm Lip}_0(A;\RR^m)$ with respect to $W^{1,p}_\mu$-norm defined in \eqref{W1pmu-norm} with 
$$
{\rm Lip}_0(A;\RR^m):=\big\{u\in{\rm Lip}(\Omega;\RR^m):u=0\hbox{ on }\Omega\setminus A\big\};
$$
\item[\rm(c)] $\Omega$ satisfies the Vitali covering theorem, i.e., for every $A\subset \Omega$ and every family $\mathcal{F}$ of closed balls in $\Omega$, if $\inf\{\rho>0:\overline{Q}_\rho(x)\in\mathcal{F}\}=0$ for all $x\in A$ then there exists a countable disjointed subfamily $\mathcal{G}$ of $\mathcal{F}$ such that $\mu(A\setminus \cup_{Q\in\mathcal{G}}Q)=0;$ in other words, $A\subset \big(\cup_{Q\in\mathcal{G}}Q\big)\cup N$ with $\mu(N)=0;$
\item[\rm(d)] for every $u\in W^{1,p}_\mu(\Omega;\RR^m)$ and $\mu$-a.e. $x\in \Omega$ there exists $u_x\in W^{1,p}_\mu(\Omega;\RR^m)$ such that{\rm:}
\begin{eqnarray}
&&\nabla_\mu u_x(y)=\nabla_\mu u(x)\hbox{ for $\mu$-a.e. $y\in \Omega$};\label{FinALAssuMpTIOnOne}\\
&&\lim_{\rho\to0}{1\over\rho^p}\mint_{Q_\rho(x)}|u(y)-u_x(y)|^pd\mu(y)=0;\label{FinALAssuMpTIOnTwo}
\end{eqnarray}
\item[\rm(e)] for every $x\in \Omega$, every $\rho>0$ and every $s\in]0,1[$ there exists a Uryshon function $\varphi\in{\rm Lip}(\Omega)$ for the pair $(\Omega\setminus Q_\rho(x),\overline{Q}_{s\rho}(x))\footnote{Given a metric space $(\Omega,d)$, by a Uryshon function from $\Omega$ to $\RR$ for the pair $(\Omega\setminus V,K)$, where $K\subset V\subset\Omega$ with $K$ compact and $V$ open, we mean a continuous function $\varphi:\Omega\to\RR$ such that $\varphi(x)\in[0,1]$ for all $x\in\Omega$, $\varphi(x)=0$ for all $x\in\Omega\setminus V$ and $\varphi(x)=1$ for all $x\in K$.}$ such that 
$$
\|D_\mu\varphi\|_{L^\infty_\mu(\Omega;\RR^N)}\leq{\alpha\over\rho(1-s)}
$$
for some $\alpha>0$.
\end{enumerate}
If moreover $(X,d)$ is a length space then 
\begin{enumerate}[leftmargin=*]
\item[\rm(f)] for $\mu$-a.e. $x\in \Omega$,
\begin{equation}\label{DoublINgAssUMpTiON}
\lim_{s\to 1^-}\liminf_{\rho\to0}{\mu(Q_{s\rho}(x))\over\mu(Q_\rho(x))}=\lim_{s\to 1^-}\limsup_{\rho\to0}{\mu(Q_{s\rho}(x))\over\mu(Q_\rho(x))}=1.
\end{equation}
\end{enumerate}
\end{proposition}

\begin{remark}\label{ReMArK-VItALi-For-OpEN-SEtS}
As $\mu$ is a Radon measure, if $\Omega$ satisfies the Vitali covering theorem, i.e., Proposition \ref{Fundamental-Proposition-for-CalcVar-in-MMS}(c) holds, then for every $A\in\mathcal{O}(\Omega)$ and every $\eps>0$ there exists a countable family $\{Q_{\rho_i}(x_i)\}_{i\in I}$ of disjoint open balls of $A$ with $x_i\in A$, $\rho_i\in]0,\eps[$ and $\mu(\partial Q_{\rho_i}(x_i))=0$ such that $\mu\big(A\setminus\cup_{i\in I}Q_{\rho_i}(x_i)\big)=0$.
\end{remark}

\begin{proof}[\bf Proof of Proposition \ref{Fundamental-Proposition-for-CalcVar-in-MMS}]
Firstly, $\Omega$ satisfies the Vitali covering theorem, i.e., the property (c) holds, because $\mu$ is doubling on $\Omega$ (see \cite[Theorem 2.8.18]{Fed-69}). Secondly, the closability of the $\mu$-gradient in ${\rm Lip}(\Omega;\RR^m)$, given by Theorem \ref{cheeger-theorem}(d), can be extended from ${\rm Lip}(\Omega;\RR^m)$ to $W^{1,p}_\mu(\Omega;\RR^m)$ by using the closability theorem of Franchi, Haj{\l}asz and Koskela (see \cite[Theorem 10]{fran-haj-kos99}). Thus, the property (a) is satisfied. Thirdly,  according to Cheeger (see \cite[\S 4, p. 450]{cheeger99} and also \cite{Haj-Kos-95,hajlaszkoskela00}), since $\mu$ is doubling on $\Omega$ and $\Omega$ supports a weak $(1,p)$-Poincar\'e inequality, we can assert that there exist $c>0$ and $\chi>1$ such that for every $0<\rho\leq\rho_0$, with $\rho_0\geq 0$, every $v\in W^{1,p}_{\mu,0}(\Omega;\RR^m)$ and every $p$-weak upper gradient $g\in L^p_\mu(\Omega;\RR^m)$ for $v$,
\begin{equation}\label{Cheeger-Coro-Eq1}
\left(\int_{Q_\rho(x)}|v|^{\chi p}d\mu\right)^{1\over\chi p}\leq \rho c\left(\int_{Q_\rho(x)}|g|^pd\mu\right)^{1\over p}.
\end{equation}
On the other hand, from Cheeger (see \cite[Theorems 2.10 and 2.18]{cheeger99}), for each $w\in W^{1,p}_\mu(\Omega)$ there exists a unique $p$-weak upper gradient for $w$, denoted by $g_w\in L ^p_\mu(\Omega)$ and called the minimal $p$-weak upper gradient for $w$, such that for every  $p$-weak upper gradient $g\in L^p_\mu(\Omega)$ for $w$, $g_w(x)\leq g(x)$ for $\mu$-a.e.  $x\in \Omega$. Moreover (see \cite[\S 4]{cheeger99} and also \cite[\S B.2, p. 363]{bjorn-bjorn-11}, \cite{bjorn00} and \cite[Remark 2.15]{gong-haj-13}), there exists $\alpha\geq 1$ such that for every $w\in W^{1,p}_\mu(\Omega)$ and $\mu$-a.e. $x\in \Omega$, 
$$
{1\over\alpha} |g_w(x)|\leq|D_\mu w(x)|\leq\alpha|g_w(x)|.
$$ 
As for $v=(v_i)_{i=1,\cdots,m}\in W^{1,p}_\mu(\Omega;\RR^m)$ we have $\nabla_\mu v=(D_\mu v_i)_{i=1,\cdots,m}$, it follows that 
\begin{equation}\label{Cheeger-Coro-Eq2}
{1\over \alpha} |g_v(x)|\leq|\nabla_\mu v(x)|\leq\alpha|g_v(x)|
\end{equation}
for $\mu$-a.e. $x\in \Omega$, where $g_v:=(g_{v_i})_{i=1,\cdots,m}$ is naturally called the minimal $p$-weak upper gradient for $v$. Combining \eqref{Cheeger-Coro-Eq1} with \eqref{Cheeger-Coro-Eq2} we obtain the property (b). Fourthly,  from Bj{\"o}rn (see \cite[Theorem 4.5 and Corollary 4.6]{bjorn00} and also \cite[Theorem 2.12]{gong-haj-13}) we see that for every $\alpha$, every $u\in W^{1,p}_\mu(\Omega;\RR^m)$ and $\mu$-a.e. $x\in \Omega_\alpha$,
$$
\nabla_\mu u_x(y)=\nabla_\mu u(x)\hbox{ for }\mu\hbox{-a.a. }y\in \Omega_\alpha,
$$
where $u_x\in W^{1,p}_\mu(\Omega;\RR^m)$ is given by 
$$
u_x(y):=u(y)-u(x)-\nabla_\mu u(x)\cdot(\xi^\alpha(y)-\xi^\alpha(x))
$$
 and $u$ is $L^p_\mu$-differentiable at $x$, i.e.,
$$
\lim_{\rho\to 0}{1\over\rho}\|u(y)-u_x(y)\|_{L^p_\mu(Q_\rho(x);\RR^m)}=0.
$$
Hence the property (d) is verified. Fifthly, given $\rho>0$, $s\in]0,1[$ and $x\in \Omega$, there exists a Uryshon function $\varphi\in {\rm Lip}(\Omega)$ for the pair $(\Omega\setminus Q_\rho(x)), \overline{Q}_{s\rho}(x))$ such 
$$
\|{\rm Lip}\varphi\|_{L^\infty_\mu(\Omega)}\leq{1\over\rho(1-s)},
$$
where for every $y\in \Omega$,
$$
{\rm Lip}\varphi(y):=\limsup_{d(y,z)\to0}{|\varphi(y)-\varphi(z)|\over d(y,z)}.
$$
But, since $\mu$ is doubling on $\Omega$ and $\Omega$ supports a weak $(1,p)$-Poincar\'e inequality, from Cheeger (see \cite[Theorem 6.1]{cheeger99}) we have ${\rm Lip}\varphi(y)=g_\varphi(y)$ for $\mu$-a.e. $y\in \Omega$, where $g_{\varphi}$ is the minimal $p$-weak upper gradient for $\varphi$. Hence 
$$
\|D_\mu\varphi\|_{L^\infty_\mu(\Omega;\RR^N)}\leq{\alpha\over\rho(1-s)}
$$ 
because $|D_\mu\varphi(y)|\leq\alpha|g_{\varphi}(y)|$ for $\mu$-a.e. $y\in \Omega$. Consequently the property (e) holds. Finally, if moreover $(X,d)$ is a length space then so is $(\Omega,d)$. Thus, from Colding and Minicozzi II (see \cite{colding-minicozzi98} and \cite[Proposition 6.12]{cheeger99}) we can assert that there exists $\beta>0$ such that for every $x\in \Omega$, every $\rho>0$ and every $s\in]0,1[$, 
$$
\mu(Q_\rho(x)\setminus Q_{s\rho}(x))\leq 2^\beta(1-s)^\beta\mu(Q_\rho(x)),
$$
which implies the property (f).
\end{proof}

\subsection{The De Giorgi-Letta lemma}

Let $\Omega=(\Omega,d)$ be a metric space, let $\mathcal{O}(\Omega)$ be the class of open subsets of $\Omega$ and let $\mathcal{B}(\Omega)$ be the class of Borel subsets of $\Omega$, i.e., the smallest $\sigma$-algebra containing the open (or equivalently the closed) subsets of $\Omega$. The following result is due to De Giorgi and Letta (see \cite{degiorgi-letta77} and also \cite[Lemma 3.3.6 p. 105]{buttazzo89}).

\begin{lemma}\label{DeGiorgi-Letta-Lemma}
Let $\mathcal{S}:\mathcal{O}(\Omega)\to[0,\infty]$ be an increasing set function, i.e., $\mathcal{S}(A)\leq \mathcal{S}(B)$ for all $A,B\in\mathcal{O}(\Omega)$ such $A\subset B$, satisfying the following four conditions{\rm:}
\begin{enumerate}[leftmargin=*]
\item[{\rm(a)}] $\mathcal{S}(\emptyset)=0;$
\item[{\rm(b)}] $\mathcal{S}$ is superadditive, i.e., $\mathcal{S}(A\cup B)\geq \mathcal{S}(A)+\mathcal{S}(B)$ for all $A,B\in\mathcal{O}(\Omega)$ such that $A\cap B=\emptyset;$
\item[{\rm(c)}] $\mathcal{S}$ is subadditive, i.e., $\mathcal{S}(A\cup B)\leq \mathcal{S}(A)+\mathcal{S}(B)$ for all $A,B\in\mathcal{O}(\Omega);$
\item[{\rm(d)}] there exists a finite Radon measure $\nu$ on $\Omega$ such that $\mathcal{S}(A)\leq\nu(A)$ for all $A\in\mathcal{O}(\Omega)$.
\end{enumerate}
Then, $\mathcal{S}$ can be uniquely extended to a finite positive Radon measure on $\Omega$ which is absolutely continuous with respect to $\nu$.
\end{lemma}


\section{Proof of the $\Gamma$-convergence theorem}

This section is devoted to the proof of Theorem \ref{MainTheorem} which is divided into five steps.

\medskip

\paragraph{\bf Step 1: integral representation of the \boldmath$\Gamma$\unboldmath-limit inf and the \boldmath$\Gamma$\unboldmath-limit sup} 

For each $u\in W^{1,p}_\mu(\Omega;\RR^m)$ we consider the set functions $\mathcal{S}^-_u,\mathcal{S}^+_u:\mathcal{O}(\Omega)\to[0,\infty]$ given by:
\begin{trivlist}
\item $\displaystyle \mathcal{S}^-_u(A):=\Gamma(L^p_\mu)\hbox{-}\liminf_{t\to\infty}E_t(u,A)$;
\item $\displaystyle \mathcal{S}^+_u(A):=\Gamma(L^p_\mu)\hbox{-}\limsup_{t\to\infty}E_t(u,A)$.
\end{trivlist}
\begin{lemma}\label{Lemma-1-MT}
If \eqref{Hyp1} holds then{\rm:}
\begin{trivlist}
\item $\displaystyle \mathcal{S}_u^-(A)=\int_A\lambda_u^-(x)d\mu(x);$
\item $\displaystyle \mathcal{S}_u^+(A)=\int_A\lambda_u^+(x)d\mu(x)$
\end{trivlist}
for all $u\in W^{1,p}_\mu(\Omega;\RR^m)$ and all $A\in\mathcal{O}(\Omega)$ with $\lambda^-_u,\lambda^+_u\in L^1_\mu(\Omega)$ given by{\rm:}
\begin{trivlist}
\item $\displaystyle\lambda^-_u(x)=\lim_{\rho\to 0}{\mathcal{S}^-_u(Q_\rho(x))\over\mu(Q_\rho(x))};$
\item $\displaystyle\lambda^+_u(x)=\lim_{\rho\to 0}{\mathcal{S}^+_u(Q_\rho(x))\over\mu(Q_\rho(x))}$.
\end{trivlist}
\end{lemma}

\begin{proof}[\bf Proof of Lemma \ref{Lemma-1-MT}]
Fix $u\in W^{1,p}_\mu(\Omega;\RR^m)$.  Using the right inequality in \eqref{Hyp1} we see that
\begin{eqnarray}
&&\mathcal{S}^-_u(A)\leq\int_A\beta(1+|\nabla_\mu u(x)|^p)d\mu(x)\hbox{ for all }A\in\mathcal{O}(\Omega)\label{ImPorTanT-EqUAtION}\\
&& \hbox{\big(resp. }\mathcal{S}^+_u(A)\leq\int_A\beta(1+|\nabla_\mu u(x)|^p)d\mu(x)\hbox{ for all }A\in\mathcal{O}(\Omega)\big).\label{ImPorTanT-EqUAtION-bis}
\end{eqnarray}
Thus, the condition (d) of Lemma \ref{DeGiorgi-Letta-Lemma} is satisfied with $\nu=\beta(1+|\nabla_\mu u|^p)d\mu$ (which is absolutely continuous with respect to $\mu$). On the other hand, it is easily seen that the conditions (a) and (b) of Lemma \ref{DeGiorgi-Letta-Lemma} are satisfied. Hence, the proof is completed by proving  the condition (c) of Lemma \ref{DeGiorgi-Letta-Lemma}, i.e., 
\begin{eqnarray}
&&\mathcal{S}^-_u(A\cup B)\leq \mathcal{S}^-_u(A)+\mathcal{S}^-_u(B)\hbox{ for all }A,B\in\mathcal{O}(\Omega)\label{Subadditivity-First-Goal}\\
&& \hbox{\big(resp. }\mathcal{S}^+_u(A\cup B)\leq \mathcal{S}^+_u(A)+\mathcal{S}^+_u(B)\hbox{ for all }A,B\in\mathcal{O}(\Omega)\big).\label{Subadditivity-First-Goal-bis}
\end{eqnarray}

Indeed, by Lemma \ref{DeGiorgi-Letta-Lemma}, the set function $\mathcal{S}^-_u$ (resp. $\mathcal{S}^+_u$) can be (uniquely) extended to a (finite) positive Radon measure which is absolutely continuous with respect to $\mu$, and the theorem follows by using Radon-Nikodym's theorem and then Lebesgue's differentiation theorem. 

\begin{remark}\label{Remark-Lemma-1-MT}
Lemma \ref{Lemma-1-MT} shows that both $\Gamma(L^p_\mu)\hbox{-}\liminf_{t\to\infty}E_t(u,\cdot)$ and $\Gamma(L^p_\mu)\hbox{-}\limsup_{t\to\infty}E_t(u,\cdot)$ can be uniquely extended to a finite positive Radon measure on $\Omega$ which is absolutely continuous with respect to $\mu$.
\end{remark}

To show \eqref{Subadditivity-First-Goal} (resp. \eqref{Subadditivity-First-Goal-bis}) we need the following lemma.
\begin{lemma}\label{LeMMa-MaiN-TheOReM1}
If $U,V,Z,T\in\mathcal{O}(\Omega)$ are such that $\overline{Z}\subset U$ and $T\subset V$,  then
\begin{eqnarray}
&&\mathcal{S}^-_u(Z\cup T)\leq\mathcal{S}^-_u(U)+\mathcal{S}^-_u(V)\label{SubAddiTive-Goal}\\
&& \hbox{\big(resp. }\mathcal{S}^+_u(Z\cup T)\leq\mathcal{S}^+_u(U)+\mathcal{S}^+_u(V)\big).\label{SubAddiTive-Goal-bis}
\end{eqnarray}
\end{lemma}
\begin{proof}[\bf Proof of Lemma \ref{LeMMa-MaiN-TheOReM1}]
As the proof of \eqref{SubAddiTive-Goal} and \eqref{SubAddiTive-Goal-bis} are exactly the same, we will only prove \eqref{SubAddiTive-Goal}. Let $\{u_t\}_{t>0}$ and $\{v_t\}_{t>0}$ be two sequences in $W^{1,p}_\mu(\Omega;\RR^m)$ such that:
\begin{eqnarray}
&& u_t\to u\hbox{ in }L^p_\mu(\Omega;\RR^m);\label{PrOoF-MT2-EquA1}\\
&& v_t\to u\hbox{ in }L^p_\mu(\Omega;\RR^m);\label{PrOoF-MT2-EquA2}\\
&& \lim_{t\to\infty}\int_UL_t(x,\nabla_\mu u_t(x))d\mu(x)=\mathcal{S}^-_u(U)<\infty;\label{PrOoF-MT2-EquA3}\\
&& \lim_{t\to\infty}\int_VL_t(x,\nabla_\mu v_t(x))d\mu(x)=\mathcal{S}^-_u(V)<\infty.\label{PrOoF-MT2-EquA4}
\end{eqnarray}
Fix $\delta\in]0,{\rm dist}(Z,\partial U)[$ with $\partial U:=\overline{U}\setminus U$, fix any $t>0$ and any $q\geq 1$ and consider $W^-_i,W^+_i\subset \Omega$ given by:
\begin{trivlist} 
\item[]$W^-_i:=\left\{x\in \Omega:{\rm dist}(x,Z)\leq {\delta\over 3}+{(i-1)\delta\over 3q}\right\}$;
\item[]$W^+_i:=\left\{x\in \Omega:{\delta\over 3}+{i\delta\over 3q}\leq{\rm dist}(x,Z)\right\},$
\end{trivlist}
where $i\in\{1,\cdots,q\}$. For every $i\in\{1,\cdots,q\}$ there exists a Uryshon function $\varphi_i\in{\rm Lip}(\Omega)$ for the pair $(W^+_i,W^-_i)$. Define $w_t^i\in W^{1,p}_\mu(\Omega;\RR^m)$ by 
$$
w^i_t:=\varphi_iu_t+(1-\varphi_i)v_t.
$$ 
Setting $W_i:=\Omega\setminus (W^-_i\cup W^{+}_i)$ and using Theorem \ref{cheeger-theorem}(d) and \eqref{Mu-der-Prod} we have
$$
\nabla_\mu w_t^i=\left\{
\begin{array}{ll}
\nabla_\mu u_t&\hbox{in }W^-_i\\
D_\mu\varphi_i\otimes(u_t-v_t)+\varphi_i\nabla_\mu u_t+(1-\varphi_i)\nabla_\mu v_t&\hbox{in }W_i\\
\nabla_\mu v_t&\hbox{in }W^+_i.
\end{array}
\right.
$$
Noticing that $Z\cup T=((Z\cup T)\cap W^-_i)\cup(W\cap W_i)\cup(T\cap W^+_i)$ with $(Z\cup T)\cap W_i^-\subset U$, $T\cap W^+_i\subset V$ and $W:=T\cap\{x\in U:{\delta\over 3}<{\rm dist}(x,Z)<{2\delta\over 3}\}$ we deduce that 
\begin{eqnarray}
\int_{Z\cup T}L_t(x,\nabla_\mu w^i_t)d\mu&\leq&\int_UL_t(x,\nabla_\mu u_t)d\mu+\int_VL_t(x,\nabla_\mu v_t)d\mu\label{LayERs-Eq1}\\
&&+\int_{W\cap W_i}L_t(x,\nabla_\mu w^i_t)d\mu\nonumber
\end{eqnarray}
for all $i\in\{1,\cdots,q\}$. Moreover, from the right inequality in \eqref{Hyp1} we see that for each $i\in\{1,\cdots,q\}$,
\begin{eqnarray}
\int_{W\cap W_i}L_t(x,\nabla_\mu w^i_t)d\mu&\leq&c\|D_\mu\varphi_i\|^p_{L^\infty_\mu(\Omega;\RR^N)}\|u_t-v_t\|^p_{L^p_\mu(\Omega;\RR^m)}\label{LayERs-Eq2}\\
&&+c\int_{W\cap W_i}(1+|\nabla_\mu u_t|^p+|\nabla_\mu v_t|^p)d\mu\nonumber
\end{eqnarray}
with $c:=2^{2p}\beta$. Substituting \eqref{LayERs-Eq2} into \eqref{LayERs-Eq1} and averaging these inequalities, it follows that for every $t>0$ and every $q\geq 1$, there exists $i_{t,q}\in\{1,\cdots,q\}$ such that
\begin{eqnarray}
\int_{Z\cup T}L_t(x,\nabla_\mu w_t^{i_{t,q}})d\mu&\leq&\int_UL_t(x,\nabla_\mu u_t)d\mu+\int_VL_t(x,\nabla_\mu v_t)d\mu\nonumber\\
&&+{c\over q}\sum_{i=1}^q\|D_\mu\varphi_i\|^p_{L^\infty_\mu(\Omega;\RR^N)}\|u_t-v_t\|^p_{L^p_\mu(\Omega;\RR^m)}\nonumber\\
&&+{c\over q}\left(\mu(\Omega)+\int_U|\nabla_\mu u_t|^pd\mu+\int_V|\nabla_\mu v_t|^pd\mu\right).\nonumber
\end{eqnarray}
On the other hand, by \eqref{PrOoF-MT2-EquA1} and \eqref{PrOoF-MT2-EquA2} we have:
\begin{trivlist}
\item[] $\displaystyle\lim_{t\to\infty}\|u_t-v_t\|^p_{L^p_\mu(\Omega;\RR^m)}=0$;
\item[] $\displaystyle\lim_{t\to\infty}\|w_t^{i_{t,q}}-u\|^p_{L^p_\mu(\Omega;\RR^m)}=0$ for all $q\geq 1$.
\end{trivlist}
Moreover, using \eqref{PrOoF-MT2-EquA3} and \eqref{PrOoF-MT2-EquA4} together with the left inequality in \eqref{Hyp1} we see that:
\begin{trivlist}
\item[] $\displaystyle\limsup_{t\to\infty}\int_U|\nabla_\mu u_t(x)|^pd\mu(x)<\infty$;
\item[] $\displaystyle\limsup_{t\to\infty}\int_V|\nabla_\mu v_t(x)|^pd\mu(x)<\infty$.
\end{trivlist}
Letting $t\to\infty$ (and taking \eqref{PrOoF-MT2-EquA3} and \eqref{PrOoF-MT2-EquA4} into account) we deduce that for every $q\geq 1$,
\begin{equation}\label{Limit-AvErAgE-LaYerS}
\mathcal{S}^-_u(Z\cup T)\leq\liminf_{t\to\infty}\int_{Z\cup T}L_t(x,\nabla_\mu w_t^{i_{t,q}}(x))d\mu(x)\leq\mathcal{S}^-_u(U)+\mathcal{S}^-_u(V)+{\hat c\over q}
\end{equation}
with $\hat c:=c(\mu(\Omega)+\limsup_{t\to\infty}\int_U|\nabla_\mu u_t(x)|^pd\mu(x)+\limsup_{t\to\infty}\int_V|\nabla_\mu v_t(x)|^pd\mu(x))$, and \eqref{SubAddiTive-Goal} follows from \eqref{Limit-AvErAgE-LaYerS} by letting $q\to\infty$. 
\end{proof}

\medskip

We now prove \eqref{Subadditivity-First-Goal} and \eqref{Subadditivity-First-Goal-bis}. Fix $A,B\in\mathcal{O}(\Omega)$. Fix any $\eps>0$ and consider $C,D\in\mathcal{O}(\Omega)$ such that $\overline{C}\subset A$, $\overline{D}\subset B$ and 
$$
\int_E\beta(1+|\nabla_\mu u(x)|^p)d\mu(x)<\eps
$$ 
with $E:=A\cup B\setminus\overline{C\cup D}$. Then $\mathcal{S}^-_u(E)\leq\eps$ by \eqref{ImPorTanT-EqUAtION} and $\mathcal{S}^+_u(E)\leq\eps$ by \eqref{ImPorTanT-EqUAtION-bis}. Let $\hat C,\hat D\in\mathcal{O}(\Omega)$ be such that $\overline{C}\subset\hat C$, $\overline{\hat C}\subset A$, $\overline{D}\subset\hat D$ and $\overline{\hat D}\subset B$. Applying Lemma \ref{LeMMa-MaiN-TheOReM1} with $U=\hat C\cup\hat D$, $V=T=E$ and $Z=C\cup D$ (resp. $U=A$, $V=B$, $Z=\hat C$ and $T=\hat D$) we obtain:
\begin{trivlist}
\item $\mathcal{S}^-_u(A\cup B)\leq\mathcal{S}^-_u(\hat C\cup\hat D)+\eps\hbox{ \big(resp. }\mathcal{S}^-_u(\hat C\cup\hat D)\leq\mathcal{S}^-_u(A)+\mathcal{S}^-_u(B)\big)$;
\item $\mathcal{S}^+_u(A\cup B)\leq\mathcal{S}^+_u(\hat C\cup\hat D)+\eps\hbox{ \big(resp. }\mathcal{S}^+_u(\hat C\cup\hat D)\leq\mathcal{S}^+_u(A)+\mathcal{S}^+_u(B)\big)$,
\end{trivlist}
and \eqref{Subadditivity-First-Goal} and \eqref{Subadditivity-First-Goal-bis} follows by letting $\eps\to0$.
\end{proof}

\medskip

\paragraph{\bf Step 2: other formulas for the \boldmath$\Gamma$\unboldmath-limit inf and the \boldmath$\Gamma$\unboldmath-limit sup} 
Consider the variational integrals $E^-_0, E^+_0:W^{1,p}_\mu(\Omega;\RR^m)\times\mathcal{O}(\Omega)\to[0,\infty]$ given by:
\begin{trivlist}
\item $\displaystyle E^-_0(u,A):=\inf\left\{\liminf_{t\to\infty}E_t(u_t,A):W^{1,p}_{\mu,0}(A;\RR^m)\ni u_t-u\stackrel{L^p_\mu}{\to}0\right\}$;
\item $\displaystyle E^+_0(u,A):=\inf\left\{\limsup_{t\to\infty}E_t(u_t,A):W^{1,p}_{\mu,0}(A;\RR^m)\ni u_t-u\stackrel{L^p_\mu}{\to}0\right\}$.
\end{trivlist}
\begin{lemma}\label{Lemma-2-MT}
If \eqref{Hyp1} holds then{\rm:}
\begin{eqnarray}
&&\Gamma(L^p_\mu)\hbox{-}\liminf_{t\to\infty}E_t(u,A)=E_0^-(u,A);\label{NeW-FoRmUlA}\\
&&\Gamma(L^p_\mu)\hbox{-}\limsup_{t\to\infty}E_t(u,A)=E_0^+(u,A)\label{NeW-FoRmUlA-Bis}
\end{eqnarray}
for all $u\in W^{1,p}_\mu(\Omega;\RR^m)$ and all $A\in\mathcal{O}(\Omega)$.
\end{lemma}
\begin{proof}[\bf Proof of Lemma \ref{Lemma-2-MT}]
As the proof of \eqref{NeW-FoRmUlA} and \eqref{NeW-FoRmUlA-Bis} are exactly the same, we will only prove \eqref{NeW-FoRmUlA-Bis}. Fix $u\in W^{1,p}_\mu(\Omega;\RR^m)$ and $A\in\mathcal{O}(\Omega)$. Noticing that $W^{1,p}_{\mu,0}(A;\RR^m)\subset W^{1,p}_\mu(\Omega;\RR^m)$ we have $E^+_0(u;A)\geq \Gamma(L^p_\mu)\hbox{-}\limsup_{t\to\infty}E_t(u,A)$. Thus, it remains to prove that
\begin{equation}\label{GOal-Lemma-Bis}
E^+_0(u;A)\leq\Gamma(L^p_\mu)\hbox{-}\limsup_{t\to\infty}E_t(u,A).
\end{equation}
Let $\{u_t\}_{t>0}\subset W^{1,p}_\mu(\Omega;\RR^m)$ be such that
\begin{eqnarray}
&& u_t\to u\hbox{ in }L^p_\mu(\Omega;\RR^m)\label{PrOoF-MT2-EquA1-BiS};\\
&& \lim_{t\to\infty}\int_A L_t(x,\nabla_\mu u_t(x))d\mu(x)=\Gamma(L^p_\mu)\hbox{-}\limsup_{t\to\infty}E_t(u,A)<\infty.\label{PrOoF-MT2-EquA3-BiS}
\end{eqnarray}
Fix $\delta>0$ and set $A_\delta:=\{x\in A:{\rm dist}(x,\partial A)>\delta\}$ with $\partial A:=\overline{A}\setminus A$. Fix any $t>0$ and any $q\geq 1$ and consider $W^-_i,W^+_i\subset \Omega$ given by
\begin{trivlist} 
\item[]$W^-_i:=\left\{x\in \Omega:{\rm dist}(x,A_\delta)\leq {\delta\over 3}+{(i-1)\delta\over 3q}\right\}$;
\item[]$W^+_i:=\left\{x\in \Omega:{\delta\over 3}+{i\delta\over 3q}\leq{\rm dist}(x,A_\delta)\right\}$,
\end{trivlist}
where $i\in\{1,\cdots,q\}$. (Note that $W^-_i\subset A$.) For every $i\in\{1,\cdots,q\}$ there exists a Uryshon function $\varphi_i\in {\rm Lip}(\Omega)$ for the pair $(W^+_i,W^-_i)$. Define $w_t^i:X\to\RR^m$  by 
$$
w^i_t:=\varphi_iu_t+(1-\varphi_i)u.
$$ 
Then $w_t^i-u\in W^{1,p}_{\mu,0}(A;\RR^m)$. Setting $W_i:=\Omega\setminus (W^-_i\cup W^{+}_i)\subset A$ and using Theorem \ref{cheeger-theorem}(d) and \eqref{Mu-der-Prod} we have
$$
\nabla_\mu w_t^i=\left\{
\begin{array}{ll}
\nabla_\mu u_t&\hbox{in }W^-_i\\
D_\mu\varphi_i\otimes(u_t-u)+\varphi_i\nabla_\mu u_t+(1-\varphi_i)\nabla_\mu u&\hbox{in }W_i\\
\nabla_\mu u&\hbox{in }W^+_i.
\end{array}
\right.
$$
Noticing that $A=W^-_i\cup W_i\cup(A\cap W^+_i)$ we deduce that for every $i\in\{1,\cdots,q\}$,
\begin{eqnarray}
\int_A L_t(x,\nabla_\mu w^i_t)d\mu&\leq&\int_A L_t(x,\nabla_\mu u_t)d\mu+\int_{A\cap W^+_i}L_t(x,\nabla_\mu u)d\mu\label{LayERs-Eq1-BiS}\\
&&+\int_{W_i} L_t(x,\nabla_\mu w^i_t)d\mu.\nonumber
\end{eqnarray}
Moreover, from the right inequality in \eqref{Hyp1} we see that for each $i\in\{1,\cdots,q\}$,
\begin{eqnarray}
\int_{W_i} L_t(x,\nabla_\mu w^i_t)d\mu&\leq&c\|D_\mu\varphi_i\|^p_{L^\infty_\mu(\Omega;\RR^N)}\|u_t-u\|^p_{L^p_\mu(\Omega;\RR^m)}\label{LayERs-Eq2-BiS}\\
&&+c\int_{W_i}(1+|\nabla_\mu u_t|^p+|\nabla_\mu u|^p)d\mu\nonumber
\end{eqnarray}
with $c:=2^{2p}\beta$. Substituting \eqref{LayERs-Eq2-BiS} into \eqref{LayERs-Eq1-BiS} and averaging these inequalities, it follows that for every $t>0$ and every $q\geq 1$, there exists $i_{t,q}\in\{1,\cdots,q\}$ such that
\begin{eqnarray}
\int_{A}L_t(x,\nabla_\mu w_t^{i_{t,q}})d\mu&\leq&\int_AL_t(x,\nabla_\mu u_t)d\mu+{1\over q}\int_A L_t(x,\nabla_\mu u)d\mu\nonumber\\
&&+{c\over q}\sum_{i=1}^q\|D_\mu\varphi_i\|^p_{L^\infty_\mu(\Omega;\RR^N)}\|u_t-u\|^p_{L^p_\mu(\Omega;\RR^m)}\nonumber\\
&&+{c\over q}\left(\mu(A)+\int_A|\nabla_\mu u_t|^pd\mu+\int_A|\nabla_\mu u|^pd\mu\right).\nonumber
\end{eqnarray}
On the other hand, by \eqref{PrOoF-MT2-EquA1-BiS} we have
$$
\lim_{t\to\infty}\|w_t^{i_{t,q}}-u\|^p_{L^p_\mu(\Omega;\RR^m)}=0\hbox{ for all }q\geq 1.
$$
Moreover, using \eqref{PrOoF-MT2-EquA3-BiS}  together with the left inequality in \eqref{Hyp1} we see that
$$
\limsup_{t\to\infty}\int_A|\nabla_\mu u_t(x)|^pd\mu(x)<\infty.
$$
Letting $t\to\infty$ (and taking \eqref{PrOoF-MT2-EquA3-BiS} into account) we deduce that for every $q\geq 1$,
\begin{eqnarray}
E^+_0(u;A)&\leq&\limsup_{t\to\infty}\int_{A}L_t(x,\nabla_\mu w_t^{i_{t,q}})d\mu\label{Limit-AvErAgE-LaYerS-BiS}\\
&\leq&\Gamma(L^p_\mu)\hbox{-}\limsup_{t\to\infty}E_t(u,A)+{1\over q}\int_AL_t(x,\nabla_\mu u)d\mu+{\hat c\over q}\nonumber
\end{eqnarray}
with $\hat c:=\beta(\mu(A)+\limsup_{t\to\infty}\int_A|\nabla_\mu u_t(x)|^pd\mu(x)+\int_A|\nabla_\mu u(x)|^pd\mu(x))$, and \eqref{GOal-Lemma-Bis} follows from \eqref{Limit-AvErAgE-LaYerS-BiS} by letting $q\to\infty$. 
\end{proof}

\medskip

\paragraph{\bf Step 3: using the Vitali envelope} For each $u\in W^{1,p}_\mu(\Omega;\RR^m)$ we consider the set functions ${\rm \underline{m}}_u, {\rm \overline{m}}_u:\mathcal{O}(\Omega)\to[0,\infty]$ by:
\begin{trivlist}
\item$\displaystyle {\rm \underline{m}}_u(A):=\liminf_{t\to\infty}\inf\left\{E_t(v,A):v-u\in W^{1,p}_{\mu,0}(A;\RR^m)\right\}$;
\item $\displaystyle {\rm \overline{m}}_u(A):=\limsup_{t\to\infty}\inf\left\{E_t(v,A):v-u\in W^{1,p}_{\mu,0}(A;\RR^m)\right\}$.
\end{trivlist}

For each $\eps>0$ and each $A\in\mathcal{O}(\Omega)$, denote the class of countable families $\{Q_i:=Q_{\rho_i}(x_i)\}_{i\in I}$ of disjoint open balls of $A$ with $x_i\in A$, $\rho_i={\rm diam}(Q_i)\in]0,\eps[$ and $\mu(\partial Q_i)=0$ such that $\mu(A\setminus\cup_{i\in I}Q_i)=0$ by $\mathcal{V}_\eps(A)$, consider ${\rm \overline{m}}_u^\eps:\mathcal{O}(\Omega)\to[0,\infty]$ given by
$$
\displaystyle{\rm \overline{m}}_u^\eps(A):=\inf\left\{\sum_{i\in I}{\rm \overline{m}}_u(Q_i):\{Q_i\}_{i\in I}\in \mathcal{V}_\eps(A)\right\},
$$
and define ${\rm \overline{m}}^*_u:\mathcal{O}(\Omega)\to[0,\infty]$ by
$$
\displaystyle{\rm \overline{m}}^*_u(A):=\sup_{\eps>0}{\rm \overline{m}}^\eps_u(A)=\lim_{\eps\to0}{\rm \overline{m}}_u^\eps(A).
$$
The set function ${\rm \overline{m}}^*_u$ is called the Vitali envelope of ${\rm \overline{m}}_u$, see \cite[Section 3]{AHM15-PP} for more details. (Note that as $\Omega$ satisfies the Vitali covering theorem, see Proposition \ref{Fundamental-Proposition-for-CalcVar-in-MMS}(c) and Remark \ref{ReMArK-VItALi-For-OpEN-SEtS}, we have $\mathcal{V}_\eps(A)\not=\emptyset$ for all $A\in\mathcal{O}(\Omega)$ and all $\eps>0$.) 

\begin{lemma}\label{Lemma-3-MT}
If \eqref{Hyp1} holds then{\rm:}
\begin{eqnarray}
&&\Gamma(L^p_\mu)\hbox{-}\liminf_{t\to\infty}E_t(u,A)\geq {\rm \underline{m}}_u(A);\label{Eq-Lemma-3-MT}\\
&& \Gamma(L^p_\mu)\hbox{-}\limsup_{t\to\infty}E_t(u,A)={\rm \overline{m}}^*_u(A)\label{Eq-Lemma-3-MT-bis}
\end{eqnarray}
for all $u\in W^{1,p}_\mu(\Omega;\RR^m)$ and all $A\in\mathcal{O}(\Omega)$.
\end{lemma}

\begin{proof}[\bf Proof of Lemma \ref{Lemma-3-MT}]
From Lemma \ref{Lemma-2-MT} it is easy to see that $\Gamma(L^p_\mu)\hbox{-}\liminf_{t\to\infty}E_t(u,A)\geq {\rm \underline{m}}_u(A)$ and $\Gamma(L^p_\mu)\hbox{-}\limsup_{t\to\infty}E_t(u,A)\geq {\rm \overline{m}}_u(A)$ and so $\Gamma(L^p_\mu)\hbox{-}\limsup_{t\to\infty}E_t(u,A)\geq {\rm \overline{m}}^*_u(A)$ because in the proof of Lemma \ref{Lemma-1-MT} it is established that $\Gamma(L^p_\mu)\hbox{-}\limsup_{t\to\infty}E_t(u,\cdot)$ can be uniquely extended to a finite positive Radon measure on $\Omega$, see Remark \ref{Remark-Lemma-1-MT}. Hence \eqref{Eq-Lemma-3-MT} holds and, to establish \eqref{Eq-Lemma-3-MT-bis}, it remains to prove that
\begin{equation}\label{EEEEqqqq}
\Gamma(L^p_\mu)\hbox{-}\limsup_{t\to\infty}E_t(u,A)\leq {\rm \overline{m}}^*_u(A)
\end{equation}
with ${\rm \overline{m}}^*_u(A)<\infty$. Fix any $\eps>0$. Given $A\in\mathcal{O}(\Omega)$, by definition of ${\rm \overline{m}}^\eps_u(A)$, there exists $\{Q_{i}\}_{i\in I}\in\mathcal{V}_\eps(A)$  such that 
\begin{equation}\label{EEEqqq1}
\sum_{i\in I}{\rm \overline{m}}_u(Q_{i})\leq {\rm \overline{m}}^\eps_u(A)+{\eps\over 2}.
\end{equation}
Fix any $t>0$ and define ${\rm m}^t_u:\mathcal{O}(\Omega)\to[0,\infty]$ by
$$
{\rm m}^t_u(A):=\inf\left\{E_t(v,A):v-u\in W^{1,p}_{\mu,0}(A;\RR^m)\right\}.
$$
(Thus $\overline{m}_u(\cdot)=\limsup_{t\to\infty} m^t_u(\cdot)$.) Given any $i\in I$, by definition of ${\rm{m}}^t_u(Q_i)$, there exists $v_{t}^i\in W^{1,p}_{\mu}(Q_{i};\RR^m)$ such that $v_{t}^i-u\in W^{1,p}_{\mu,0}(Q_{i};\RR^m)$ and
\begin{equation}\label{EEEqqq1-bis}
E_t(v_{t}^i,Q_i)\leq {\rm {m}}^t_u(Q_i)+{\eps \mu(Q_i)\over 2\mu(A)}.
\end{equation}
Define $u^\eps_t:\Omega\to\RR^m$ by 
$$
u^\eps_t:=\left\{
\begin{array}{ll}
u&\hbox{in }\Omega\setminus A\\
v_t^i&\hbox{in }Q_{i}.
\end{array}
\right.
$$
Then $u^\eps_t-u\in W^{1,p}_{\mu,0}(A;\RR^m)$. Moreover, because of Proposition \ref{Fundamental-Proposition-for-CalcVar-in-MMS}(a), $\nabla_\mu u^\eps_t(x)=\nabla_\mu v^i_t(x)$ for $\mu$-a.e. $x\in Q_i$. From \eqref{EEEqqq1-bis} we see that
$$
E_t(u_{t}^\eps,A)\leq \sum_{i\in I}{\rm {m}}^t_u(Q_i)+{\eps\over 2},
$$
hence
$
\limsup_{t\to\infty}E_t(u_{t}^\eps,A)\leq {\rm \overline{m}}^\eps_u(A)+{\eps}
$
by using \eqref{EEEqqq1}, and consequently
\begin{equation}\label{Diag-SteP3-1}
\limsup_{\eps\to 0}\limsup_{t\to\infty}E_t(u_{t}^\eps,A)\leq {\rm \overline{m}}^*_u(A).
\end{equation}
On the other hand, we have
\begin{eqnarray*}
\|u_t^\eps-u\|^{p}_{L^{\chi p}_\mu(\Omega;\RR^m)}=\left(\int_A|u_t^\eps-u|^{\chi p}d\mu\right)^{1\over\chi}&=&\left(\sum_{i\in I}\int_{Q_{i}}|v_t^i-u|^{\chi p}d\mu\right)^{1\over\chi}\\
&\leq&\sum_{i\in I}\left(\int_{Q_{i}}|v_t^i-u|^{\chi p}d\mu\right)^{1\over\chi}
\end{eqnarray*}
with $\chi\geq 1$ given by \eqref{Poincare-Inequality}. As $\Omega$ supports a $p$-Sobolev inequality, see Proposition \ref{Fundamental-Proposition-for-CalcVar-in-MMS}(b), and ${\rm diam}(Q_i)\in]0,\eps[$ for all $i\in I$, we have
$$
\|u_t^\eps-u\|^{p}_{L^{\chi p}_\mu(\Omega;\RR^m)}\leq\eps^{p} C_S^{p}\sum_{i\in I}\int_{Q_{i}}|\nabla_\mu v_t^i-\nabla_\mu u|^pd\mu
$$
with $C_S>0$ given by \eqref{Poincare-Inequality}, and so
\begin{equation}\label{EEEqqq2}
\|u_t^\eps-u\|^{p}_{L^{\chi p}_\mu(\Omega;\RR^m)}\leq 2^p\eps^{p} C_S^p\left(\sum_{i\in I}\int_{Q_{i}}|\nabla_\mu v_t^i|^pd\mu+\int_A|\nabla_\mu u|^pd\mu\right).
\end{equation}
Taking the left inequality in \eqref{Hyp1}, \eqref{EEEqqq1-bis} and \eqref{EEEqqq1} into account, from \eqref{EEEqqq2} we deduce that
$$
\limsup_{t\to\infty}\|u_t^\eps-u\|^p_{L^{\chi p}_\mu(\Omega;\RR^m)}\leq 2^p C_S^p\eps^{p}\left({1\over \alpha}({\rm \overline{m}}^\eps_u(A)+\eps)+\int_A|\nabla_\mu u|^pd\mu\right)
$$
which gives
\begin{equation}\label{Diag-SteP3-2}
\limsup_{\eps\to0}\limsup_{t\to\infty}\|u_t^\eps-u\|^p_{L^{\chi p}_\mu(\Omega;\RR^m)}=0
\end{equation}
because $\lim_{\eps\to0}\overline{{\rm m}}_u^\eps(A)=\overline{{\rm m}}^*_u(A)<\infty$. According to \eqref{Diag-SteP3-1} and \eqref{Diag-SteP3-2}, by diagonalization there exists a mapping $t\mapsto \eps_t$, with $\eps_t\to0$ as $t\to\infty$, such that:
\begin{eqnarray}
&&\lim_{t\to\infty}\|w_t-u\|^p_{L^{\chi p}_\mu(\Omega;\RR^m)}=0;\label{EnD-EqSteP3-1}\\
&&\limsup_{t\to\infty}E_t(w_t,A)\leq {\rm \overline{m}}^*_u(A)\label{EnD-EqSteP3-2}
\end{eqnarray}
with $w_t:=u^{\eps_t}_t$. Since $\chi p\geq p$, $w_t\to u$ in $L^{p}_\mu(\Omega;\RR^m)$ by \eqref{EnD-EqSteP3-1}, and \eqref{EEEEqqqq} follows from \eqref{EnD-EqSteP3-2} by noticing that $\Gamma(L^p_\mu)\hbox{-}\limsup_{t\to\infty}E_t(u;A)\leq\limsup_{t\to\infty}E_t(w_t,A)$.
\end{proof}

\medskip

\paragraph{\bf Step 4: differentiation with respect to \boldmath$\mu$\unboldmath} First of all, using Lemma \ref{Lemma-1-MT}, Remark \ref{Remark-Lemma-1-MT} and Lemma \ref{Lemma-3-MT} it easily seen that:
\begin{eqnarray}
&&\hskip-18mm  \Gamma(L^p_\mu)\hbox{-}\liminf_{t\to\infty}E_t(u,A)\geq\int_A\limsup_{\rho\to0}{{\rm \underline{m}}_u(Q_\rho(x))\over\mu(Q_\rho(x))}d\mu(x);\label{Step4-EquaTion1}\\
&&\hskip-18mm  \Gamma(L^p_\mu)\hbox{-}\liminf_{t\to\infty}E_t(u,A)=\int_A\lim_{\rho\to0}{{\rm \overline{m}}^*_u(Q_\rho(x))\over\mu(Q_\rho(x))}d\mu(x)\geq\int_A\limsup_{\rho\to0}{{\rm \overline{m}}_u(Q_\rho(x))\over\mu(Q_\rho(x))}d\mu(x)\label{Step4-EquaTion2}
\end{eqnarray}
for all $u\in W^{1,p}_\mu(\Omega;\RR^m)$ and all $A\in\mathcal{O}(\Omega)$. Moreover, we have
\begin{lemma}\label{Lemma-4-MT}
For $\mu$-a.e. $x\in\Omega$, 
\begin{equation}\label{Step4-EquaTion3}
\lim_{\rho\to0}{{\rm \overline{m}}^*_u(Q_\rho(x))\over\mu(Q_\rho(x))}\leq\liminf_{\rho\to0}{{\rm \overline{m}}_u(Q_\rho(x))\over\mu(Q_\rho(x))}.
\end{equation}
\end{lemma}
\begin{proof}[\bf Proof of Lemma \ref{Lemma-4-MT}]
Fix any $s>0$. Denote the class of open balls $Q_\rho(x)$, with $x\in \Omega$ and $\rho>0$, such that $\overline{\rm m}^*_u(Q_{\rho}(x))>\overline{\rm m}_u(Q_{\rho}(x))+s\mu(Q_{\rho}(x))$ by $\mathcal{G}_s$ and define $N_s\subset \Omega$ by
$$
N_s:=\Big\{x\in \Omega:\forall\delta>0\ \exists\rho\in]0,\delta[\ Q_\rho(x)\in\mathcal{G}_s\Big\}.
$$
Fix any $\eps>0$. Using the definition of $N_s$, we can assert that for each $x\in N_s$ there exists $\{\rho_{x,n}\}_n\subset]0,\eps[$ with $\rho_{x,n}\to0$ as $n\to\infty$ such that for every $n\geq 1$, $\mu(\partial Q_{\rho_{x,n}}(x))=0$ and $Q_{\rho_{x,n}}(x)\in\mathcal{G}_s$. Consider the family $\mathcal{F}_0$ of closed balls in $\Omega$ given by
$$
\mathcal{F}_0:=\left\{\overline{Q}_{\rho_{x,n}}(x):x\in N_s\hbox{ and }n\geq 1\right\}.
$$
Then $\inf\left\{r>0:\overline{Q}_r(x)\in\mathcal{F}_0\right\}=0$ for all $x\in N_s$. As $\Omega$ satisfies the Vitali covering theorem, there exists a disjointed countable subfamily $\{\overline{Q}_i\}_{i\in I_0}$ of closed balls of $\mathcal{F}_0$ (with $\mu(\partial Q_i)=0$ and ${\rm diam}(Q_i)\in]0,\eps[$) such that
$$
N_s\subset\Big(\cupp_{i\in I_0}\overline{Q}_i\Big)\cup\Big(N_s\setminus\cupp_{i\in I_0}\overline{Q}_i\Big)\hbox{ with }\mu\Big(N_s\setminus\cupp_{i\in I_0}\overline{Q}_i\Big)=0. 
$$
If $\mu\big(\cup_{i\in I_0}\overline{Q}_i\big)=0$ then \eqref{Step4-EquaTion3} will follow. Indeed, in this case we have $\mu(N_s)=0$, i.e., $\mu(\Omega\setminus N_s)=\mu(\Omega)$, and given $x\in \Omega\setminus N_s$ there exists $\delta>0$ such that $\overline{\rm m}^*_u(Q_{\rho}(x))\leq\overline{\rm m}_u(Q_{\rho}(x))+s\mu(Q_{\rho}(x))$ for all $\rho\in]0,\delta[$. Hence 
$$
\lim_{\rho\to0}{\overline{\rm m}^*_u(Q_\rho(x))\over\mu(Q_\rho(x))}\leq\liminf_{\rho\to0}{\overline{\rm m}_u(Q_\rho(x))\over\mu(Q_\rho(x))}+s\hbox{ for all }s>0,
$$
and \eqref{Step4-EquaTion3} follows by letting $s\to0$. 

To establish that $\mu\big(\cupp_{i\in I_0}\overline{Q}_i\big)=0$ it is sufficient to prove that for every finite subset $J$ of $I_0$,
\begin{equation}\label{GoAl-PPPRRRoooFFF}
\mu\Big(\cupp_{i\in J}\overline{Q}_i\Big)=0.
\end{equation} 
As $\Omega$ satisfies the Vitali covering theorem and $\Omega\setminus \cupp_{i\in J}\overline{Q}_i$ is open, there exists a countable family $\{B_i\}_{i\in I}$ of disjoint open balls of $\Omega\setminus \cupp_{i\in J}\overline{Q}_i$, with $\mu(\partial B_i)=0$ and ${\rm diam}(B_i)\in]0,\eps[$, such that
\begin{equation}\label{EqUaT-VitaL-1}
\mu\left(\Big(\Omega\setminus\cupp_{i\in J}\overline{Q}_i\Big)\setminus \cupp_{i\in I}B_i\right)=\mu\left(\Omega\setminus\Big(\cupp_{i\in I}B_i\Big)\cup\Big(\cupp_{i\in J}Q_i\Big)\right)=0.
\end{equation}
Recalling that $\overline{\rm m}^*_u$ is the restriction to $\mathcal{O}(\Omega)$ of a finite positive Radon measure  which is absolutely continuous with respect to $\mu$ (see Lemmas \ref{Lemma-1-MT}, Remark \ref{Remark-Lemma-1-MT} and \ref{Lemma-3-MT}), from \eqref{EqUaT-VitaL-1} we see that
$$
\overline{\rm m}^*_u(\Omega)=\sum_{i\in I}\overline{\rm m}^*_u(B_i)+\sum_{i\in J}\overline{\rm m}^*_u(Q_i).
$$
Moreover, $Q_i\in\mathcal{G}_s$ for all $i\in J$, i.e., $\overline{\rm m}^*_u(Q_i)>\overline{\rm m}_u(Q_i)+s\mu(Q_i)$ for all $i\in J$, and $\overline{\rm m}^*_u\geq \overline{\rm m}_u$, hence 
$$
\overline{\rm m}^*_u(\Omega)\geq \sum_{i\in I}\overline{\rm m}_u(B_i)+\sum_{i\in J}\overline{\rm m}_u(Q_i)+s\mu\left(\cupp_{i\in J}Q_i\right).
$$
As $\{B_i\}_{i\in I}\cup\{Q_i\}_{i\in J}\in\mathcal{V}_\eps(\Omega)$ we have $\sum_{i\in I}\overline{\rm m}_u(B_i)+\sum_{i\in J}\overline{\rm m}_u(Q_i)\geq\overline{\rm m}_u^\eps(\Omega)$, hence $\overline{\rm m}^*_u(\Omega)\geq \overline{\rm m}^\eps_u(\Omega)+s\mu(\cupp_{i\in J}Q_i)$, and \eqref{GoAl-PPPRRRoooFFF} follows by letting $\eps\to0$.
\end{proof}

\medskip

Combining \eqref{Step4-EquaTion3} with \eqref{Step4-EquaTion2} we obtain
\begin{equation}\label{Step4-EquaTion2-bis}
 \Gamma(L^p_\mu)\hbox{-}\liminf_{t\to\infty}E_t(u,A)=\int_A\lim_{\rho\to0}{{\rm \overline{m}}_u(Q_\rho(x))\over\mu(Q_\rho(x))}d\mu(x)
\end{equation}
for all $u\in W^{1,p}(\Omega;\RR^m)$ and all $A\in\mathcal{O}(\Omega)$.

\medskip

\paragraph{\bf Step 5: removing by affine functions} According to \eqref{Step4-EquaTion1} and \eqref{Step4-EquaTion2-bis}, the proof of Theorem \ref{MainTheorem} will be completed if we prove that for each $u\in W^{1,p}_\mu(\Omega;\RR^m)$ and $\mu$-a.e. $x\in\Omega$, we have:
\begin{eqnarray}
&&\limsup_{\rho\to0}{\underline{\rm m}_u(Q_\rho(x))\over\mu(Q_\rho(x))}\geq\limsup_{\rho\to0}{\underline{\rm m}_{u_x}(Q_\rho(x))\over\mu(Q_\rho(x))};\label{FiNaL-EqUa1}\\
&&\lim_{\rho\to0}{\overline{\rm m}_u(Q_\rho(x))\over\mu(Q_\rho(x))}=\lim_{\rho\to0}{\overline{\rm m}_{u_x}(Q_\rho(x))\over\mu(Q_\rho(x))},\label{FiNaL-EqUa2}
\end{eqnarray}
where $u_x\in W^{1,p}_\mu(\Omega;\RR^m)$ is given by Proposition \ref{Fundamental-Proposition-for-CalcVar-in-MMS}(d) (and satisfies \eqref{FinALAssuMpTIOnOne} and \eqref{FinALAssuMpTIOnTwo}).

\begin{remark}
In fact, we have:
\begin{trivlist}
\item $\displaystyle {\underline{\rm m}_{u_x}(Q_\rho(x))\over\mu(Q_\rho(x))}=\liminf_{t\to\infty}\mathcal{H}^\rho_\mu L_t(x,\nabla_\mu u(x))$;
\item $\displaystyle {\overline{\rm m}_{u_x}(Q_\rho(x))\over\mu(Q_\rho(x))}=\limsup_{t\to\infty}\mathcal{H}^\rho_\mu L_t(x,\nabla_\mu u(x))$,
\end{trivlist}
where $\mathcal{H}^\rho_\mu L_t:\MM\to[0,\infty]$ is given by \eqref{t-mu-quasiconvexification}.
\end{remark}

We only give the proof of \eqref{FiNaL-EqUa1} because the equality \eqref{FiNaL-EqUa2} follows from two inequalities whose the proofs use the same method as in  \eqref{FiNaL-EqUa1}.  For each $t>0$ and each $z\in W^{1,p}_\mu(\Omega;\RR^m)$, let $m^t_z:\mathcal{O}(\Omega)\to[0,\infty]$ be given by 
$$
{\rm m}^t_z(A):=\inf\left\{E_t(w,A):w-z\in W^{1,p}_{\mu,0}(A;\RR^m)\right\},
$$
where we recall that $E_t(w,A):=\int_AL_t(x,\nabla_\mu w(x))d\mu(x)$. Note that:
\begin{trivlist}
\item $\displaystyle \underline{\rm m}_z(\cdot):=\liminf_{t\to\infty}{\rm m}_z^t(\cdot)$ 
\item  (resp. $\displaystyle\overline{\rm m}_z(\cdot):=\limsup_{t\to\infty}{\rm m}_z^t(\cdot))$. 
\end{trivlist}

\begin{proof}[\bf Proof of (\ref{FiNaL-EqUa1})]
Fix any $\eps>0$. Fix any $s\in]0,1[$ and any $\rho\in]0,\eps[$. By definition of ${\rm m}^t_u(Q_{s\rho}(x))$, where there is no loss of generality in assuming that $\mu(\partial Q_{s\rho}(x))=0$, there exists $w:\Omega\to\RR^m$ such that $w-u\in W^{1,p}_{\mu,0}(Q_{s\rho}(x);\RR^m)$ and 
\begin{equation}\label{FunDaMenTal-IneQualiTY-Final}
\int_{Q_{s\rho}(x)} L_t(y,\nabla_\mu w(y))d\mu(y)\leq {\rm m}^t_u(Q_{s\rho}(x))+\eps\mu(Q_{s\rho}(x)).
\end{equation}
From Proposition \ref{Fundamental-Proposition-for-CalcVar-in-MMS}(e) there exists a Uryshon function $\varphi\in{\rm Lip}(\Omega)$ for the pair $(\Omega\setminus Q_{\rho}(x),\overline{Q}_{s\rho}(x))$ such that 
\begin{equation}\label{PlAtEAu-FunCtIOn-ProPerTy}
\|D_\mu\varphi\|_{L^\infty_\mu(\Omega;\RR^N)}\leq {\gamma\over\rho(1-s)}
\end{equation} 
for some $\gamma>0$ (which does not depend on $\rho$). Define $v\in W^{1,p}_{\mu}(Q_\rho(x);\RR^m)$ by
$$
v:=\varphi u+(1-\varphi)u_x.
$$
Then $v-u_x\in W^{1,p}_{\mu,0}(Q_\rho(x);\RR^m)$. Using Theorem \ref{cheeger-theorem}(d) and \eqref{Mu-der-Prod} we have
$$
\nabla_\mu v=\left\{
\begin{array}{ll}
\nabla_{\mu}u&\hbox{in }\overline{Q}_{s\rho}(x)\\
D_\mu\varphi\otimes(u-u_x)+\varphi\nabla_\mu u+(1-\varphi)\nabla_\mu u(x)&\hbox{in }Q_\rho(x)\setminus \overline{Q}_{s\rho}(x).
\end{array}
\right.
$$
As $w-u\in W^{1,p}_{\mu,0}(Q_{s\rho}(x);\RR^m)$ we have $v+(w-u)-u_x\in W^{1,p}_{\mu,0}(Q_\rho(x);\RR^m)$. Noticing that $\mu(\partial Q_{s\rho}(x))=0$ and, because of Proposition \eqref{Fundamental-Proposition-for-CalcVar-in-MMS}(a), $\nabla_\mu (w-u)(y)=0$ for $\mu$-a.e. $y\in Q_\rho(x)\setminus \overline{Q}_{s\rho}(x)$ and taking \eqref{FunDaMenTal-IneQualiTY-Final}, the right inequality in \eqref{Hyp1} and \eqref{PlAtEAu-FunCtIOn-ProPerTy} into account we deduce that
\begin{eqnarray*}
{{\rm m}^t_{u_x}(Q_{\rho}(x))\over\mu(Q_{s\rho}(x))}&\leq&{1\over \mu(Q_{s\rho}(x))}\int_{Q_\rho(x)}L_t(y,\nabla_\mu v+\nabla_\mu (w-u))d\mu\label{FiRsTEquAtIOnofTHelaSTPrOOf}\\
&=&{1\over\mu(Q_{s\rho}(x))}\int_{\overline{Q}_{s\rho}(x)}L_t(y,\nabla_\mu u+\nabla_\mu (w-u))d\mu\nonumber\\
&&+{1\over\mu(Q_{s\rho}(x))}\int_{Q_\rho(x)\setminus \overline{Q}_{s\rho}(x)}L_t(y,\nabla_\mu v)d\mu\nonumber\\
&\leq &{{\rm m}^t_{u}(Q_{s\rho}(x))\over\mu(Q_{s\rho}(x))}+\eps\nonumber\\
&& +2^{2p}\beta\left({\gamma^p\over(1-s)^p}{\mu(Q_\rho(x))\over \mu(Q_{s\rho}(x))}{1\over\rho^p}\mint_{Q_\rho(x)}|u-u_x|^pd\mu+{A_{\rho,s}\over\mu(Q_{s\rho}(x))}\right)\nonumber
\end{eqnarray*}
with
$$
A_{\rho,s}:=\mu(Q_\rho(x)\setminus Q_{s\rho}(x))|\nabla_\mu u(x)|^p+\int_{Q_\rho(x)\setminus Q_{s\rho}(x)}|\nabla_\mu u|^pd\mu.
$$
Thus, noticing that $\mu(Q_\rho(x))\geq \mu(Q_{s\rho}(x))$ and letting $t\to\infty$, we obtain
\begin{eqnarray}
{\underline{\rm m}_{u_x}(Q_\rho(x))\over\mu(Q_\rho(x))}&\leq &{\underline{\rm m}_u(Q_{s\rho}(x))\over\mu(Q_{s\rho}(x))}+\eps\label{FiRsTEquAtIOnofTHelaSTPrOOf}\\
&&+2^{2p}\beta\left({\gamma^p\over(1-s)^p}{\mu(Q_\rho(x))\over \mu(Q_{s\rho}(x))}{1\over\rho^p}\mint_{Q_\rho(x)}|u-u_x|^pd\mu+{A_{\rho,s}\over\mu(Q_{s\rho}(x))}\right).\nonumber
\end{eqnarray}
On the other hand, as $\mu$ is a doubling measure we can assert that
$$
\lim_{r\to0}\mint_{Q_r(x)}\big||\nabla_\mu u(y)|^p-|\nabla_\mu u(x)|^p\big|d\mu(y)=0.
$$
But
\begin{eqnarray*}
{A_{\rho,s}\over\mu(Q_{s\rho}(x))}&\leq& 2\left({\mu(Q_{\rho}(x))\over\mu(Q_{s\rho}(x))}-1\right)|\nabla_\mu u(x)|^p\\
&&+{\mu(Q_\rho(x))\over \mu(Q_{s\rho}(x))}\mint_{Q_\rho(x)}\big||\nabla_\mu u(y)|^p-|\nabla_\mu u(x)|^p\big|d\mu(y)
\end{eqnarray*}
and so
\begin{equation}\label{LiMItiNrhOofArhot}
\limsup_{\rho\to0}{A_{\rho,s}\over\mu(Q_{s\rho}(x))}\leq  2\left(\limsup_{\rho\to0}{\mu(Q_{\rho}(x)\over\mu(Q_{s\rho}(x))}-1\right)|\nabla_\mu u(x)|^p. 
\end{equation}
Letting $\rho\to 0$ in \eqref{FiRsTEquAtIOnofTHelaSTPrOOf} and using \eqref{FinALAssuMpTIOnTwo} and  \eqref{LiMItiNrhOofArhot} we see that
\begin{eqnarray*}
\limsup_{\rho\to0}{\underline{\rm m}_{u_x}(Q_\rho(x))\over\mu(Q_\rho(x))}&\leq& \limsup_{\rho\to0}{\underline{\rm m}_{u}(Q_{s\rho}(x))\over\mu(Q_{s\rho}(x))}+\eps+2\left(\limsup_{\rho\to0}{\mu(Q_{\rho}(x)\over\mu(Q_{s\rho}(x))}-1\right)|\nabla_\mu u(x)|^p\\
&=& \limsup_{\rho\to0}{\underline{\rm m}_{u}(Q_{\rho}(x))\over\mu(Q_{\rho}(x))}+\eps+2\left(\limsup_{\rho\to0}{\mu(Q_{\rho}(x)\over\mu(Q_{s\rho}(x))}-1\right)|\nabla_\mu u(x)|^p.
\end{eqnarray*}
Letting $s\to 1$ and using \eqref{DoublINgAssUMpTiON} we conclude that
$$
\limsup_{\rho\to0}{\underline{\rm m}_{u_x}(Q_\rho(x))\over\mu(Q_\rho(x))}\leq \limsup_{\rho\to0}{\underline{\rm m}_{u}(Q_\rho(x))\over\mu(Q_\rho(x))}+\eps
$$
and \eqref{FiNaL-EqUa1} follows by letting $\eps\to0$.
\end{proof}


\section{Proof of homogenization theorems}

This section is devoted to the proof of Theorems \ref{Coro-Homogenization} and \ref{Coro-Homogenization2}. We begin by proving Theorem \ref{ST-MMspace}.

\begin{proof}[\bf Proof of Theorem \ref{ST-MMspace}]
Fix $Q\in \mathfrak{S}(X)$.

\paragraph{\bf Case 1: $(X,d,\mu)$  is assumed to be a meshable $(G,\{h_t\}_{t>0})$-metric measure space which is asymptotically periodic with respect to $\mathfrak{S}(X)$}

Fix $k\in\NN^*$ and consider $t_{Q,k}>0$ given by Definition \ref{Def-appli-hom-2}. To each $t\geq t_{Q,k}$ there correspond $k^-_t,k^+_t\in\NN^*$ and $g^-_t,g^+_t\in G$ such that \eqref{Hyp-Sub-2} and \eqref{Hyp-Sub-1} hold. Fix any $t\geq t_{k,Q}$. Taking the left inclusion in \eqref{Hyp-Sub-2} into account, we see that
$$
h_t(Q)=g^-_t{\hskip0.3mm\rm o\hskip0.3mm}h_{kk^-_t}(\UU)\cup\left(h_t(Q)\setminus g^-_t{\hskip0.3mm\rm o\hskip0.3mm}h_{kk^-_t}(\UU)\right).
$$
As $\mathcal{S}$ is subadditive and $G$-invariant, it follows that 
\begin{equation}\label{SubProof-part-1-A}
\mathcal{S}\left(h_t(Q)\right)\leq\mathcal{S}\left(h_{kk^-_t}(\UU)\right)+\mathcal{S}\left(h_t(Q)\setminus g^-_t{\hskip0.3mm\rm o\hskip0.3mm}h_{kk^-_t}(\UU)\right).
\end{equation}
Taking the right inclusion in \eqref{Hyp-Sub-2} into account, it is easily seen that
$$
h_t(Q)\setminus g^-_t{\hskip0.3mm\rm o\hskip0.3mm}h_{kk^-_t}(\UU)\subset g^+_t{\hskip0.3mm\rm o\hskip0.3mm}h_{kk^+_t}(\UU)\setminus g^-_t{\hskip0.3mm\rm o\hskip0.3mm}h_{kk^-_t}(\UU),
$$ 
hence
$$
\mathcal{S}\left(h_t(Q)\setminus g^-_t{\hskip0.3mm\rm o\hskip0.3mm}h_{kk^-_t}(\UU)\right)\leq c\left(\mu\left(g^+_t{\hskip0.3mm\rm o\hskip0.3mm}h_{kk^+_t}(\UU)\right)-\mu\left(g^-_t{\hskip0.3mm\rm o\hskip0.3mm}h_{kk^-_t}(\UU)\right)\right)
$$
with $c>0$ given by  \eqref{Sub4}, and so
$$
\mathcal{S}\left(h_t(Q)\setminus g^-_t{\hskip0.3mm\rm o\hskip0.3mm}h_{kk^-_t}(\UU)\right)\leq c\left(\mu\left(h_{kk^+_t}(\UU)\right)-\mu\left(h_{kk^-_t}(\UU)\right)\right)
$$
because $\mu$ is $G$-invariant. From \eqref{Eq-sub-im} and \eqref{Eq-sub-1-bis} it follows that
\begin{equation}\label{SubProof-part-1-B}
\mathcal{S}\left(h_t(Q)\setminus g^-_t{\hskip0.3mm\rm o\hskip0.3mm}h_{kk^-_t}(\UU)\right)\leq c\mu(h_k(\UU))\big[\mu(h_{k^+_t}(\UU))-\mu(h_{k^-_t}(\UU))\big].
\end{equation}
Moreover, since $\mathcal{S}$ is subadditive and $G$-invariant, taking \eqref{Eq-sub-6} and \eqref{Equa-Sub-1-Bis-Bis} into account, we can assert that
\begin{equation}\label{SubProof-part-1-C}
\mathcal{S}(h_{kk^+_t}(\UU))\leq\sum_{g\in G^k_{k^+_t}}\mathcal{S}\left(g{\hskip0.3mm\rm o\hskip0.3mm}h_k(\UU)\right)=\mu\big(h_{k^+_t}(\UU)\big)\mathcal{S}\left(h_k(\UU)\right).
\end{equation}
From \eqref{SubProof-part-1-A}, \eqref{SubProof-part-1-B} and \eqref{SubProof-part-1-C} we deduce that
$$
\mathcal{S}\left(h_t(Q)\right)\leq \mu\big(h_{k^+_t}(\UU)\big)\mathcal{S}\left(h_k(\UU)\right)+c\mu(h_k(\UU))\big[\mu(h_{k^+_t}(\UU))-\mu(h_{k^-_t}(\UU))\big].
$$
As $\mu$ is $G$-invariant, from the left inclusion in \eqref{Hyp-Sub-2} and \eqref{Eq-sub-1-bis} we see that 
$$
\mu(h_t(Q))\geq \mu(h_k(\UU))\mu(h_{k^-_t}(\UU)).
$$
Hence
$$
{\mathcal{S}\left(h_t(Q)\right)\over\mu\left(h_t(Q)\right)}\leq {\mu(h_{k^+_t}(\UU))\over\mu(h_{k^-_t}(\UU))}{\mathcal{S}\left(h_k(\UU)\right)\over\mu(h_k(\UU))}+c\left({\mu(h_{k^+_t}(\UU))\over\mu(h_{k^-_t}(\UU))}-1\right).
$$
Letting $t\to\infty$ and using \eqref{Hyp-Sub-1}, and then passing to the infimum on $k$, we obtain
$$
\limsup_{t\to\infty}{\mathcal{S}\left(h_t(Q)\right)\over\mu\left(h_t(Q)\right)}\leq \inf_{k\in\NN^*}{\mathcal{S}\left(h_k(\UU)\right)\over\mu(h_k(\UU))}.
$$
Consider now $t_{1,Q}>0$ given by Definition \ref{Def-appli-hom-2} with $k=1$. Taking the right inclusion in \eqref{Hyp-Sub-2} (with $k=1$) into account, we see that
$$
g^+_t{\hskip0.3mm\rm o\hskip0.3mm}h_{k^+_t}(\UU)=h_t(Q)\cup\left(g^+_t{\hskip0.3mm\rm o\hskip0.3mm}h_{k^+_t}(\UU)\setminus h_t(Q)\right). 
$$
As $\mathcal{S}$ is subadditive and $G$-invariant, it follows that 
\begin{equation}\label{SubProof-part-1-D}
\mathcal{S}(h_{k^+_t}(\UU))\leq\mathcal{S}\left(h_t(Q)\right)+\mathcal{S}\left(g^+_t{\hskip0.3mm\rm o\hskip0.3mm}h_{k^+_t}(\UU)\setminus h_t(Q)\right).
\end{equation}
By \eqref{Hyp-Sub-2} (with $k=1$) we have
$$
g^+_t{\hskip0.3mm\rm o\hskip0.3mm}h_{k^+_t}(\UU)\setminus h_t(Q)\subset g^+_t{\hskip0.3mm\rm o\hskip0.3mm}h_{k^+_t}(\UU)\setminus g^-_t{\hskip0.3mm\rm o\hskip0.3mm}h_{k^-_t}(\UU),
$$ 
and using \eqref{Sub4} we obtain
\begin{equation}\label{SubProof-part-1-E}
\mathcal{S}\left(g^+_t{\hskip0.3mm\rm o\hskip0.3mm}h_{k^+_t}(\UU)\setminus h_t(Q)\right)\leq c\big(\mu(h_{k^+_t}(\UU))-\mu(h_{k^-_t}(\UU))\big).
\end{equation}
From \eqref{SubProof-part-1-D} and \eqref{SubProof-part-1-E} we deduce that
$$
\mathcal{S}(h_{k^+_t}(\UU))\leq\mathcal{S}\left(h_t(Q)\right)+c\big(\mu(h_{k^+_t}(\UU))-\mu(h_{k^-_t}(\UU))\big),
$$
Since $\mu$ is $G$-invariant, from the right inequality in \eqref{Hyp-Sub-2} (with $k=1$), we have 
$$
\mu(h_t(Q))\leq \mu(h_{k^+_t}(\UU)).
$$ 
Hence
$$
\inf_{k\in\NN^*}{\mathcal{S}\left(h_k(\UU)\right)\over\mu(h_k(\UU))}\leq {\mathcal{S}(h_{k^+_t}(\UU))\over\mu(h_{k^+_t}(\UU))}\leq{\mathcal{S}\left(h_t(Q)\right)\over\mu\left(h_t(Q)\right)}+c\left(1-{\mu(h_{k^-_t}(\UU))\over \mu(h_{k^+_t}(\UU))}\right).
$$
Letting $t\to\infty$ and using \eqref{Hyp-Sub-1}, we obtain 
$$
\inf_{k\in\NN^*}{\mathcal{S}\left(h_k(\UU)\right)\over\mu(h_k(\UU))}\leq \liminf_{t\to\infty}{\mathcal{S}\left(h_t(Q)\right)\over\mu\left(h_t(Q)\right)},
$$
and the proof of case 1 is complete.
\smallskip
\paragraph{\bf Case 2:  $(X,d,\mu)$ is assumed to be a strongly meshable $(G,\{h_t\}_{t>0})$-metric measure space which is weakly asymptotically periodic with respect to $\mathfrak{S}(X)$} Fix any $k\in\NN^*$ and any $t>0$ and set:
\begin{trivlist}
\item $\displaystyle Q^-_{t,k}:=\cupp\limits_{g\in G^-_{t,k}}g{\hskip0.3mm\rm o\hskip0.3mm}h_k(\UU)$;
\item $\displaystyle Q^+_{t,k}:=\cupp\limits_{g\in G^+_{t,k}}g{\hskip0.3mm\rm o\hskip0.3mm}h_k(\UU)$,
\end{trivlist}
where $G^-_{t,k}$ are $G^+_{t,k}$ are given by Definition \ref{New-def-Revised-version}. By the left inclusion in \eqref{Hyp-Sub-1-revised} we have $Q^-_{t,k}\subset h_t(Q)$ and so $h_t(Q)=Q^-_{t,k}\cup\big(h_t(Q)\setminus Q^-_{t,k}\big)$. Hence
$$
\mathcal{S}(h_t(Q))\leq \mathcal{S}\left(Q^-_{t,k}\right)+\mathcal{S}\left(h_t(Q)\setminus Q^-_{t,k}\right),
$$
and consequently
$$
{\mathcal{S}(h_t(Q))\over\mu(h_t(Q))}\leq {\mathcal{S}\left(Q^-_{t,k}\right)\over\mu\left(Q^-_{t,k}\right)}{\mu\left(Q^-_{t,k}\right)\over\mu(h_t(Q))}+{\mathcal{S}\left(h_t(Q)\setminus Q^-_{t,k}\right)\over\mu(h_t(Q))}.
$$
As $\mathcal{S}$ is subadditive and $G$-invariant (resp. $\mu$ is $G$-invariant) we have
\begin{eqnarray*}
&&\mathcal{S}\left(Q^-_{t,k}\right)\leq {\rm card}\big(G^-_{t,k}\big)\mathcal{S}(h_k(\UU))\\
&&\big(\hbox{resp. }\mu\left(Q^-_{t,k}\right)={\rm card}\big(G^-_{t,k}\big)\mu(h_k(\UU))\big).
\end{eqnarray*}
Moreover, $h_t(Q)\subset Q^+_{t,k}$ by the right inclusion in \eqref{Hyp-Sub-1-revised} which implies that $h_t(Q)\setminus Q^-_{t,k}\subset Q^+_{t,k}\setminus Q^-_{t,k}$ and so
$$
\mathcal{S}\left(h_t(Q)\setminus Q^-_{t,k}\right)\leq c\mu\left(Q^+_{t,k}\setminus Q^-_{t,k}\right)
$$
with $c>0$ given by \eqref{Sub4}. It follows that
\begin{eqnarray*}
{\mathcal{S}(h_t(Q))\over\mu(h_t(Q))}&\leq& {\mathcal{S}\left(h_k(\UU)\right)\over\mu\left(h_k(\UU)\right)}{\mu\left(Q^-_{t,k}\right)\over\mu(h_t(Q))}+{c\mu\left(Q^+_{t,k}\setminus Q^-_{t,k}\right)\over\mu(h_t(Q))}\\
&\leq&{\mathcal{S}\left(h_k(\UU)\right)\over\mu\left(h_k(\UU)\right)}+{c\mu\left(Q^+_{t,k}\setminus Q^-_{t,k}\right)\over\mu(h_t(Q))}
\end{eqnarray*}
because $\mu\left(Q^-_{t,k}\right)\leq \mu(h_t(Q))$ since $Q^-_{t,k}\subset h_t(Q)$. Letting $t\to\infty$ and using \eqref{Hyp-Sub-3-revised}, and then passing to the infimum on $k$, we obtain
$$
\limsup_{t\to\infty}{\mathcal{S}(h_t(Q))\over\mu(h_t(Q))}\leq \inf_{k\in\NN^*}{\mathcal{S}\left(h_k(\UU)\right)\over\mu\left(h_k(\UU)\right)}.
$$
We now prove that
\begin{eqnarray}\label{Final-aim-ST-revised-version}
\inf_{k\in\NN^*}{\mathcal{S}\left(h_k(\UU)\right)\over\mu\left(h_k(\UU)\right)}\leq\liminf_{t\to\infty}{\mathcal{S}(h_t(Q))\over\mu(h_t(Q))}.
\end{eqnarray}
 Fix any $t>0$. As $h_t(Q)\subset Q^+_{t,1}:=\cupp_{g\in G^+_{t,1}}g(\UU)$ by the right inclusion in \eqref{Hyp-Sub-1-revised} we have $Q^+_{t,1}=h_t(Q)\cup\big(Q^+_{t,1}\setminus h_t(Q)\big)$, and so
$$
\mathcal{S}\left(Q^+_{t,1}\right)\leq \mathcal{S}\left(h_t(Q)\right)+\mathcal{S}\left(Q^+_{t,1}\setminus h_t(Q)\right)
$$
because $\mathcal{S}$ is subadditive. But, as $\cupp_{g\in G^-_{t,1}}g(\UU)=:Q^-_{t,1}\subset h_t(Q)$ by the left inclusion in \eqref{Hyp-Sub-1-revised}, we have $Q^+_{t,1}\setminus h_t(Q)\subset Q^+_{t,1}\setminus Q^-_{t,1}$, hence
\begin{equation}\label{ST-Case2-Eq1-revised-version}
{\mathcal{S}\left(Q^+_{t,1}\right)\over \mu\left(Q^+_{t,1}\right)}\leq{\mathcal{S}\left(Q^+_{t,1}\right)\over \mu\left(h_t(Q)\right)}\leq {\mathcal{S}\left(h_t(Q)\right)\over \mu\left(h_t(Q)\right)}+{c\mu\left(Q^+_{t,1}\setminus Q^-_{t,1}\right)\over \mu\left(h_t(Q)\right)}
\end{equation}
by using \eqref{Sub4}. Consider the subclass $\mathcal{K}(X)$ of $\mathcal{B}_0(X)$ given by
$$
\mathcal{K}(X):=\bigg\{\cupp\limits_{g\in H}g(\UU):H\subset G\hbox{ and }{\rm card}(H)<\infty\bigg\}
$$
and define  the set function $\mathcal{S}_1:\mathcal{K}(X)\to]-\infty,0]$ by
$$
\mathcal{S}_1(K):=\mathcal{S}(K)-{\rm card}(H)\mathcal{S}(\UU)
$$
with $K=\cupp_{g\in H}g(\UU)$. Taking the assertion (b) of Definition \ref{New-def-Revised-version-0} into account, as $\mathcal{S}$ is subadditive and $G$-invariant, it is easily seen that $\mathcal{S}_1$ is decreasing, i.e., for every $K\in\mathcal{K}(X)$ and every $K^\prime\in\mathcal{K}(X)$,
\begin{equation}\label{Revised-Version-Eq-Sub-ST-1}
K\subset K^\prime\hbox{ implies } \mathcal{S}_1\left(K\right)\geq \mathcal{S}_1\left(K^\prime\right).
\end{equation}
Noticing that $Q^+_{t,1}\in \mathcal{K}(X)$, as $\mu$ is $G$-invariant we can assert that
\begin{equation}\label{Revised-Version-Eq-Sub-ST-2}
{\mathcal{S}\left(Q^+_{t,1}\right)\over \mu\left(Q^+_{t,1}\right)}={\mathcal{S}_1\left(Q^+_{t,1}\right)\over \mu\left(Q^+_{t,1}\right)}+{\mathcal{S}(\UU)\over\mu(\UU)}.
\end{equation}
On the other hand, by the assertion (d) of Definition \ref{New-def-Revised-version-0} (with $H=G^+_{t,1}\subset G$) there exist $i_t\in\NN^*$ and $f_t\in G$ such that
$$
Q^+_{t,1}\subset f_t{\hskip0.3mm\rm o\hskip0.3mm}h_{i_t}(\UU).
$$
Thus, using \eqref{Revised-Version-Eq-Sub-ST-1}, from \eqref{Revised-Version-Eq-Sub-ST-2} we obtain
\begin{equation}\label{Revised-Version-Eq-Sub-ST-3}
{\mathcal{S}\left(Q^+_{t,1}\right)\over \mu\left(Q^+_{t,1}\right)}\geq{\mathcal{S}_1\left(f_t{\hskip0.3mm\rm o\hskip0.3mm}h_{i_t}(\UU)\right)\over \mu\left(f_t{\hskip0.3mm\rm o\hskip0.3mm}h_{i_t}(\UU)\right)}+{\mathcal{S}(\UU)\over\mu(\UU)}\geq \inf_{(f,i)\in G\times\NN^*}{\mathcal{S}_1\left(f{\hskip0.3mm\rm o\hskip0.3mm}h_{i}(\UU)\right)\over \mu\left(f{\hskip0.3mm\rm o\hskip0.3mm}h_{i}(\UU)\right)}+{\mathcal{S}(\UU)\over\mu(\UU)}.
\end{equation}
But, using the assertion (c) of Definition \ref{New-def-Revised-version-0}, we see that for each $f\in G$ and each $i\in\NN^*$ we have $f{\hskip0.3mm\rm o\hskip0.3mm}h_{i}(\UU)=\cupp_{g\in G_{i}(f)}g(\UU)$. So, as $\mathcal{S}$ and $\mu$ are $G$-invariant, we get
\begin{eqnarray*}
{{\mathcal{S}_1\left(f{\hskip0.3mm\rm o\hskip0.3mm}h_{i}(\UU)\right)\over \mu\left(f{\hskip0.3mm\rm o\hskip0.3mm}h_{i}(\UU)\right)}}&=&{{\mathcal{S}\left(f{\hskip0.3mm\rm o\hskip0.3mm}h_{i}(\UU)\right)\over \mu\left(f{\hskip0.3mm\rm o\hskip0.3mm}h_{i}(\UU)\right)}}-{\rm card}(G_{i}(f)){\mathcal{S}(\UU)\over\mu\left(f{\hskip0.3mm\rm o\hskip0.3mm}h_{i}(\UU)\right)}\\
&\geq&{{\mathcal{S}\left(h_{i}(\UU)\right)\over \mu\left(h_{i}(\UU)\right)}}-{\mathcal{S}(\UU)\over\mu(\UU)}\\
&\geq&\inf_{k\in\NN^*}{{\mathcal{S}\left(h_{k}(\UU)\right)\over \mu\left(h_{k}(\UU)\right)}}-{\mathcal{S}(\UU)\over\mu(\UU)}
\end{eqnarray*}
for all $f\in G$ and all $i\in\NN^*$, and consequently
\begin{equation}\label{Revised-Version-Eq-Sub-ST-4}
\inf_{(f,i)\in G\times\NN^*}{\mathcal{S}_1\left(f{\hskip0.3mm\rm o\hskip0.3mm}h_{i}(\UU)\right)\over \mu\left(f{\hskip0.3mm\rm o\hskip0.3mm}h_{i}(\UU)\right)}\geq \inf_{k\in\NN^*}{{\mathcal{S}\left(h_{k}(\UU)\right)\over \mu\left(h_{k}(\UU)\right)}}-{\mathcal{S}(\UU)\over\mu(\UU)}.
\end{equation}
Combining \eqref{ST-Case2-Eq1-revised-version} with \eqref{Revised-Version-Eq-Sub-ST-3} and with \eqref{Revised-Version-Eq-Sub-ST-4}, we deduce that 
$$
\inf_{k\in\NN^*}{{\mathcal{S}\left(h_{k}(\UU)\right)\over \mu\left(h_{k}(\UU)\right)}}\leq{\mathcal{S}\left(h_t(Q)\right)\over \mu\left(h_t(Q)\right)}+{c\mu\left(Q^+_{t,1}\setminus Q^-_{t,1}\right)\over \mu\left(h_t(Q)\right)},
$$
and \eqref{Final-aim-ST-revised-version} follows by letting $t\to\infty$ and using \eqref{Hyp-Sub-3-revised}. 
\end{proof}

\begin{proof}[\bf Proof of Theorem \ref{Coro-Homogenization}] The proof consists of applying Corollary \ref{Coro-MainResult}. For this, it suffices to verify that \eqref{Equality-Coro} is satisfied. 

For each $\xi\in\MM$, we consider the set function $\mathcal{S}^\xi:\mathcal{B}_0(X)\to[0,\infty]$ defined by
$$
\mathcal{S}^\xi(A):=\inf\left\{\int_{\mathring{A}}L(y,\xi+\nabla_\mu w(y))d\mu(y):w\in W^{1,p}_{\mu,0}\big(\mathring{A};\RR^m\big)\right\}.
$$
As $\{L_t\}_{t>0}$ is a family of $(G,\{h_t\}_{t>0})$-periodic integrands modelled on $L$ (see Definition \ref{mu-PeriOdiciTY-Def}), we have
\begin{eqnarray*}
\mathcal{S}^\xi\left(h_t(Q)\right)&=&\inf\left\{\int_{h_t(Q)}L(y,\xi+\nabla_\mu w(y))d\mu(y):w\in W^{1,p}_{\mu,0}(h_t(Q);\RR^m)\right\}\\
&=&\inf\left\{\int_{Q}L(h_t(y),\xi+\nabla_\mu w(h_t(y)))d(h_t^\sharp\mu)(y):w\in W^{1,p}_{\mu,0}(h_t(Q);\RR^m)\right\}\\
&=&\mu(h_t(Q))\inf\left\{\mint_{Q}L_t(y,\xi+\nabla_\mu w(y))d\mu(y):w\in W^{1,p}_{\mu,0}(Q;\RR^m)\right\}
\end{eqnarray*}
for all $Q\in{\rm Ba}(X)$ and all $t>0$, and so: 
\begin{trivlist}
\item $\displaystyle\liminf_{t\to\infty}\mathcal{H}^\rho_\mu L_t(x,\xi)=\liminf_{t\to\infty}{\mathcal{S}^\xi\left(h_t(Q_\rho(x)\right)\over\mu\left(h_t(Q_\rho(x)\right)}$;
\item $\displaystyle\limsup_{t\to\infty}\mathcal{H}^\rho_\mu L_t(x,\xi)=\limsup_{t\to\infty}{\mathcal{S}^\xi\left(h_t(Q_\rho(x)\right)\over\mu\left(h_t(Q_\rho(x)\right)}$
\end{trivlist}
for $\mu$-a.e. $x\in \Omega$, all $\rho>0$ and all $\xi\in\MM$. But, from the second inequality in \eqref{Hyp1}, it is easy to see that $\mathcal{S}^\xi(A)\leq c\mu\big(\mathring{A}\big)\leq c\mu(A)$ for all $A\in\mathcal{B}_0(X)$, where $c:=\beta(1+|\xi|^p)$, and moreover the set function  $\mathcal{S}^\xi$ is clearly $G$-invariant and subadditive because, for each $A, B\in\mathcal{B}_0(X)$, $\mu\Big(\mathring{\widehat{A \cup B}}\setminus(\mathring{A}\cup\mathring{B})\Big)=0$ since $\mathring{\widehat{A \cup B}}\setminus(\mathring{A}\cup\mathring{B})\subset \partial A\cup\partial B$ and $\mu(\partial A)=\mu(\partial B)=0$. Thus, by Theorem \ref{ST-MMspace} we see that 
 $$
 \lim_{t\to\infty}{\mathcal{S}^\xi\left(h_t(Q_\rho(x)\right)\over\mu\left(h_t(Q_\rho(x)\right)}=\inf_{k\in\NN^*}{S^\xi\left(h_k(\UU)\right)\over\mu(h_k(\UU))}=L_{\rm hom}(\xi),
 $$
which means that $
\liminf_{t\to\infty}\mathcal{H}^\rho_\mu L_t(x,\xi)=\limsup_{t\to\infty}\mathcal{H}^\rho_\mu L_t(x,\xi)=L_{\rm hom}(\xi)
$ 
for $\mu$-a.e. $x\in \Omega$, all $\rho>0$ and all $\xi\in\MM$, i.e., \eqref{Equality-Coro} holds, and finishes the proof. 
\end{proof}

\begin{proof}[\bf Proof of Theorem \ref{Coro-Homogenization2}]
Under the hypotheses of Theorem \ref{Coro-Homogenization2} it is easy to see that, by using Theorem \ref{MainTheorem}, we have:
\begin{trivlist}
\item $\displaystyle\Gamma(L^p_\mu)\hbox{-}\displaystyle\liminf\limits_{t\to\infty} E_t(u;A)\geq\sum_{i\in I}\int_{\Omega_i\cap A}\limsup_{\rho\to 0}\liminf_{t\to\infty}\mathcal{H}^\rho_\mu L_t^i(x,\nabla_\mu u(x))d\mu(x)$;
\item $\displaystyle\Gamma(L^p_\mu)\hbox{-}\displaystyle\limsup\limits_{t\to\infty} E_t(u;A)=\sum_{i\in I}\int_{\Omega_i\cap A}\lim_{\rho\to 0}\limsup_{t\to\infty}\mathcal{H}^\rho_\mu L_t^i(x,\nabla_\mu u(x))d\mu(x)$
\end{trivlist}
for all $u\in W^{1,p}(\Omega;\RR^m)$ and all $A\in\mathcal{O}(\Omega)$. Under these hypotheses, it is also easily seen that Theorem \ref{ST-MMspace} implies that for each $i\in I$, 
$$
\liminf_{t\to\infty}\mathcal{H}^\rho_\mu L_t^i(x,\xi)=\limsup_{t\to\infty}\mathcal{H}^\rho_\mu L_t^i(x,\xi)=L^i_{\rm hom}(\xi)
$$
for $\mu$-a.e. $x\in \Omega_i\cap A$, all $\rho>0$ and all $\xi\in\MM$ with $L_{\rm hom}^i$ given by \eqref{HomoGenizaTion-Formula-LDS}, which gives the result.
\end{proof}


\end{document}